\documentclass{amsart}
\usepackage{amscd,amsmath,amssymb,amsfonts,verbatim, xcolor,graphicx}
\usepackage[all]{xy}
\allowdisplaybreaks
\usepackage{calrsfs}
\DeclareMathAlphabet{\pazocal}{OMS}{zplm}{m}{n}
\makeatletter
\newcommand*\dotp{\mathpalette\dotp@{.5}}
\newcommand*\dotp@[2]{\mathbin{\vcenter{\hbox{\scalebox{#2}{$\m@th#1\bullet$}}}}}
\makeatother

\newtheorem{theorem}{Theorem}[subsection]
\newtheorem{corollary}[theorem]{Corollary}
\newtheorem{lemma}[theorem]{Lemma}  
\newtheorem{proposition}[theorem]{Proposition}
\theoremstyle{definition}
\newtheorem{definition}[theorem]{Definition}

\theoremstyle{remark}
\newtheorem{remark}[theorem]{Remark}

\numberwithin{equation}{subsection}
\makeatletter
\newcommand*\bigcdot{\mathpalette\bigcdot@{.5}}
\newcommand*\bigcdot@[2]{\mathbin{\vcenter{\hbox{\scalebox{#2}{$\m@th#1\bullet$}}}}}
\makeatother

\title{Infinitesimal dilogarithm   on curves over truncated polynomial rings}
\author{S\.{I}nan \"{U}nver}
\address{Ko\c{c} University, Mathematics Department. Rumelifeneri Yolu, 34450, Istanbul, Turkey}
\email{sunver@ku.edu.tr}

\begin{document}
\maketitle
\noindent

\begin{abstract}
Let  $C$ be a smooth and projective curve over the truncated polynomial ring $k_m:=k[t]/(t^m), $ where $k$ is a field of characteristic 0. Using a candidate for the motivic cohomology group ${\rm H}^{3}_{\pazocal{M}}(C,\mathbb{Q}(3))$ based on the Bloch complex of weight 3,  we construct regulators  to $k$ for every $m<r<2m.$ 
Specializing this  construction, we obtain an invariant $\rho_{m,r}(f \wedge g \wedge h)$ of   rational functions $f,$ $g$ and $h$  on $C.$ The current work is a twofold generalization of our work on the infinitesimal Chow dilogarithm: we sheafify the previous construction and therefore do not restrict ourselves  to triples of rational functions and  we  construct the regulator for any $m<r<2m,$ rather than only for $m=2.$   We also define  regulators of cycles, which we expect to give a complete set of invariants for the infinitesimal part of  ${\rm CH}^{2}(k_{m},3). $ This generalizes Park's work, where the additive Chow cycles, namely the case of cycles close to 0, is handled for $r=m+1.$ In this paper, we generalize the reciprocity theorem  to pairs of cycles which are the same modulo $(t^m)$ and for any $m<r<2m.$ We expect the theory of the paper to give regulators on categories of motives over rings with nilpotents. 
 
 \end{abstract}

\section{Introduction}

In this paper, we continue our project started in \cite{unv3} and followed up  in \cite{unv4}, which aims to give infinitesimal analogs of  real analytic regulators. We  first briefly explain the classical analog of our construction and then continue with  explaining the contents of the paper. 

If $X/\mathbb{C}$ is a smooth projective curve, then by the conjectural Leray-Serre spectral sequence for motivic sheaves one would expect a map 
$$
K_{3}(X)^{(3)} _{\mathbb{Q}}={\rm H}^{3} _{\pazocal{M}}(X,\mathbb{Q}(3)) \to {\rm H}^{1} _{\pazocal{M}}(\mathbb{C},\mathbb{Q}(2))=K_{3}(\mathbb{C})^{(2)} _{\mathbb{Q}}. 
$$
Composing with  the Borel regulator $K_{3}(\mathbb{C})^{(2)} _{\mathbb{Q}} \to \mathbb{C}/(2\pi i)^{2}\mathbb{Q}$ and taking the imaginary part would give a map $K_{3}(X)^{(3)} _{\mathbb{Q}} \to \mathbb{R}.$ Up to normalization, this map can be constructed as follows \cite[\textsection 6]{arak}. 
For  $f_{1},\, f_2, \,$ and $f_3 \in \mathbb{C}(X) ^{\times},$ let  
\begin{eqnarray*}
    r_{2}(f_{1},f_{2},f_{3}):=   
 {\rm Alt}_{3}(\frac{1}{6} \log |f_1| \cdot d \log |f_{2}| \wedge d \log |f_{3}| -\frac{1}{2}\log |f_1|\cdot d \arg f_{2} \wedge d \arg f_{3} ),
\end{eqnarray*}
(s.t. $dr_{2}(f_1,f_2,f_3) =Re(d \log (f_1)\wedge d \log (f_2) \wedge d \log (f_3))$). The Chow dilogarithm map 
$
\rho: \Lambda^{3} \mathbb{C}(X) ^{\times} \to \mathbb{R}
$ is given in terms of this by 
\begin{eqnarray*}
\rho(f_1 \wedge f_2 \wedge f_3) := \int_{X(\mathbb{C})} r_{2}(f_{1},f_{2},f_{3}).
\end{eqnarray*}

In the special case when $X=\mathbb{P}^1,$ we have  $\rho(1-z,z,z-a)=D_{2}(a),$ where  $D_{2}(z):={\rm Im}(\ell i _{2}(z))+{\rm arg}(1-z)\cdot \log (|z|)$ is the Bloch-Wigner dilogarithm, with $\ell i_{2} (z)$  the (multi-valued) analytic continuation of $\sum _{1 \leq n} \frac{z^n}{n^2}.$   

The middle cohomology of 
\begin{eqnarray}\label{b3complex}
 \to B_{2}(\mathbb{C}(X)) \otimes \mathbb{C}(X) ^{\times } _{\mathbb{Q}} \to  (\oplus _{x \in X} B_{2}(\mathbb{C})) \oplus \Lambda ^3 \mathbb{C}(X) ^{\times} _{\mathbb{Q}} \to \oplus _{x \in X} \Lambda ^2 \mathbb{C} ^{\times}  _{\mathbb{Q}}\to 
\end{eqnarray}
is ${\rm H}^{3} _{\mathcal{M}}(X,\mathbb{Q}(3))\simeq K_{3}(X)_{\mathbb{Q}}^{(3)}.$ Combining $D_{2}$ and $\rho$, if we let 
$$
\rho_{X}:=-(\oplus _{x \in X} D_{2}) \oplus \rho: (\oplus _{x \in X} B_{2}(\mathbb{C})) \oplus \Lambda ^3 \mathbb{C}(X) ^{\times} _{\mathbb{Q}} \to \mathbb{R},  
$$ then $\rho_{X}$ 
vanishes on the image of $ B_{2}(\mathbb{C}(X)) \otimes  \mathbb{C}(X) ^{\times } _{\mathbb{Q}}$ and  induces the map $K_{3}(X)_{\mathbb{Q}} ^{(3)} \to \mathbb{R},$ we were looking for above.  If one assumes a theory of motivic sheaves then this is the composition  of ${\rm H}^{3} _{\mathcal{M}}(X,\mathbb{Q}(3))\to {\rm Ext}^{1} _{\mathcal{M}_{\mathbb{C}}}(\mathbb{Q}(0),{\rm H}^{2}(X/\mathbb{C})(3) )= {\rm H}^1 _{\mathcal{M}}(\mathbb{C},\mathbb{Q}(2))\to B_{2}(\mathbb{C})$ and the Bloch-Wigner dilogarithm $D_{2}:B_{2}(\mathbb{C})\to \mathbb{R}.$ 

We will be interested in the case where $C/k_{m}$ is a smooth and projective curve. We denote the underlying reduced scheme of $C$ by  $\underline{C}.$  In order to state our result in the most general setting, we need an analog of the complex (\ref{b3complex}). This construction will be based on a choice $\mathcal{P}$ of {\it smooth liftings} of the closed points $|C|$ of $\underline{C}.$ For different choices of liftings, we expect the complexes to be isomorphic in the derived category of complexes of sheaves. We will consider sheaves of functions on $C$ which will satisfy certain regularity conditions with respect to $\mathfrak{c} \in \mathcal{P}.$ We will call such a function {\it good} with respect to $\mathfrak{c}.$   Imposing this condition will allow us to define the residue of such a function along  $\mathfrak{c}.$


For each $2\leq m<r<2m,$ our  regulator will be induced from the  corresponding map  on the degree 3 cohomology of the following two term complex of sheaves:
$$
 B_{2}(\pazocal{O}_C, \underline{\mathcal{P}}) \otimes (\pazocal{O}_{C},\underline{\mathcal{P}}) ^{\times}\to \oplus _{c \in |C|}i_{c*}(B_{2}(k(\mathfrak{c}))) \oplus \Lambda ^3 (\pazocal{O}_{C},\underline{\mathcal{P}}) ^{\times}
$$
concentrated in $[2,3].$ In this  complex, $(\pazocal{O}_{C},\underline{\mathcal{P}}) ^{\times}$ is the sheaf whose sections on an open set $U$ are those elements of $\pazocal{O}_{C,\eta} ^{\times}$ which are $\mathfrak{c}$-good for $c \in \underline{U};$  $ B_{2}(\pazocal{O}_C, \underline{\mathcal{P}})$ is the sheaf associated to the presheaf whose sections on $U$ are elements of the Bloch group on $U$ which are $\mathfrak{c}$-good for $c \in \underline{U}.$     This construction is the precise analog of the construction of $\rho_{X}$ in the complex setting. Let us explain this below.

 Suppose that we are given a Zariski open cover $\{U_{i}\}_{i \in I}$ of $\underline{C}$ and a corresponding  cocyle $\gamma,$ given by the following data:   $\gamma_{i} \in \Lambda ^3 (\pazocal{O}_{C},\underline{\mathcal{P}}) ^{\times}(U_i);$ $\varepsilon_{i,c} \in B_{2}(k(\mathfrak{c})) $ for every $c \in U_{i};$ and  
 $\beta_{ij} \in  (B_{2}(\pazocal{O}_C, \underline{\mathcal{P}}) \otimes (\pazocal{O}_{C},\underline{\mathcal{P}}) ^{\times})(U_{ij}).$  We will define $\rho_{m,r}(\gamma) \in k, $ by first making many choices and then showing that the construction is independent of all the choices. 

(i) Let $\tilde{\pazocal{A}}_{\eta}/k_{\infty}:=k[[t]]$ be a smooth lifting of $\pazocal{O}_{C,\eta}$ and for every $c \in |C|,$ let $\tilde{\pazocal{A}}_{c}/k_{\infty}$ be a smooth lifting of the completion $\hat{\pazocal{O}}_{C,c}$ of the local ring of $C$ at $c,$ together with a smooth lifting $\tilde{\mathfrak{c}}$ of $\mathfrak{c}.$

(ii) Let an  $i \in I$ be arbitrary and for each $c$ choose a $j_{c} \in I$ such that $c \in U_{j_c}$ 

(iii) Choose an arbitrary lifting $\tilde{\gamma}_{i\eta} \in \Lambda^3 \tilde{\pazocal{A}}_{\eta} ^{\times}$  of the germ $\gamma _{i\eta} \in \Lambda ^3\pazocal{O}_{C,\eta}^{\times}$   

(iii) Choose a good lifting $\tilde{\gamma}_{j_c} \in \Lambda^3(\tilde{\pazocal{A}}_c,\tilde{\mathfrak{c}})^{\times}$  of the image $\hat{\gamma}_{j_c,c}$ of  $\gamma_{j_c}$ in $\Lambda^{3}(\hat{\pazocal{O}}_{C,c},\mathfrak{c})^{\times},$ for every $c \in |C|,$  

(iv) Choose an arbitrary lifting   $\tilde{\beta}_{j_c i,\eta} \in B_{2}(\tilde{\pazocal{A}}_{\eta})\otimes \tilde{\pazocal{A}}_{\eta}$  of the image  $\beta_{j_c i,\eta}\in B_{2}(\pazocal{O}_{C,\eta})\otimes \pazocal{O}_{C,\eta} ^{\times}$ of   $\beta_{j_c i},$ for every $c \in |C|.$

We then  define the value of the regulator $\rho_{m,r}$ on the above element by the expression
\begin{align}\label{firstformrho}
\rho_{m,r}(\gamma):=   \sum _{c \in |C|}{\rm Tr}_{k}\big(\ell _{m,r}(res_{\tilde{\mathfrak{c}}}\tilde{\gamma}_{j_c})-\ell i _{m,r} (\varepsilon_{j_c,c})+res_{c}\omega_{m,r}(\tilde{\gamma}_{i\eta}-\delta(\tilde{\beta}_{j_ci,\eta}),\tilde{\gamma}_{j_c})\big). 
\end{align}
We continue with the description of this expression. 

The starting point of this paper is our construction of the additive dilogarithm in \cite{unv1}. For a regular local $\mathbb{Q}$-algebra $R,$ letting $R_{m}:= R[t]/(t^m),$ for every $2\leq m <r<2m,$ we have an additive dilogarithm map $\ell i_{m,r}:B_{2}(R_{m}) \to R$ that satisfies all the analogous properties of the Bloch-Wigner dilogarithm function. Most importantly, the direct sums of these maps over all the possible $r$'s give an isomorphism between the infinitesimal part of the $K$-group $K_{3}(R_{m})^{(2)} _{\mathbb{Q}}$ and $\oplus _{m<r<2m}R.$ We  explain this in detail in \textsection 2 and give explicit formulas for these functions $\ell i_{m,r}.$ The function $\ell i_{m,r}$ can also be described in terms of the differential $\delta$ in the Bloch complex of $B_{2}(R_{\infty}),$ with $R_{\infty}:=R[[t]],$ by the following   commutative diagram 
$$
\xymatrix{
B_{2}(R_{\infty}) \ar^{\delta}[r] \ar[d] &\Lambda ^{2}R_{\infty}^{\times} \ar^{\ell _{m,r}}[d]\\
B_{2}(R_{m})\ar^{\ell i_{m,r}}[r] & R,
}
$$
where $\ell _{m,r}$ is given explicitly in Definition 
\ref{defnlmr} below. 

We can then describe the first two terms in (\ref{firstformrho}) as follows. For a connected, \'{e}tale  $k_{m}$-algebra (resp. $k_{\infty}$-algebra) $A,$ there is a canonical isomorphism $A \simeq k'_{m}$ (resp. $A\simeq k'_{\infty}$). Using this isomorphism for $k(\mathfrak{c}),$ we get a canonical identification $B_{2}(k(\mathfrak{c}))=B_{2}(k(c)_{m}).$ Therefore, $\ell i_{m,r}(\varepsilon _{j,c})\in k(c)$  is unambiguously defined using the map 
$\ell i_{m,r}: B_{2}(k(c)_{m})\to k(c).$ Since the element $\tilde{\gamma}_{j_c} \in \Lambda^3(\tilde{\pazocal{A}}_c,\tilde{\mathfrak{c}})^{\times}$ is assumed to be $\tilde{\mathfrak{c}}$-good, the residue $res_{\tilde{\mathfrak{c}}}\tilde{\gamma}_{j_c}$ is defined as an element of $\Lambda^{2}k(\tilde{\mathfrak{c}})^{\times}$  in the beginning of  \textsection 6. Using the identification $\Lambda^2k(\tilde{\mathfrak{c}})^{\times}=\Lambda^2 k(c)_{\infty}^{\times}$ and the map $\ell_{m,r}: \Lambda ^{2}k(c)_{\infty} ^{\times} \to k(c), $ we define the element $\ell _{m,r}(res_{\tilde{\mathfrak{c}}}\tilde{\gamma}_{j_c}) \in k(c).$   

Defining the last term $res_{c}\omega_{m,r}$ and proving its properties will constitute a large proportion of the paper.  For any local $\mathbb{Q}$-algebra $R,$ we define a map $L_{m,r}:B_{2}(R_m) \otimes R_{m} ^{\times} \to \Omega^{1} _{R},$ by an explicit formula in (\ref{eqn lmr}). This should be thought of as an absolute notion and does not require that $R$ be of dimension 1 over a field $k.$   If  $R/k$ is a smooth,  $k$-algebra of  dimension 1, using  $L_{m,r}$ we  construct a map $\omega_{m,r}:\Lambda ^{3}(R_{r},(t^m))^{\times}\to \Omega^{1} _{R/k}.$ Here  $(R_{r},(t^m))^{\times}$ denotes $\{(a,b)| a, \, b \in R_r ^{\times},\; ab^{-1} \in 1+(t^m) \}.$ Since we do not fix a lifting of our curve  in the construction of $\rho_{m,r}, $ defining $\omega_{m,r}$ on this group is not enough. More precisely, we need to extend $\omega_{m,r}$ to the following context. Suppose that  $\pazocal{R}$ and $\pazocal{R}'$ are smooth of relative dimension 1 over $k_{r}$ together with a fixed isomorphism: 
$$
\chi:\pazocal{R}/(t^m)\to 
\pazocal{R}'/(t^m),
$$
of $k_{m}$-algebras  between their reductions modulo $(t^m).$  Let 
$$(\pazocal{R},\pazocal{R}',\chi)^{\times} :=\{(a,b)| a \in \pazocal{R}^{\times}, \, b \in \pazocal{R}'^{\times},\; \chi(a+(t^m))=b+(t^m)  \, \, {\rm in}\,\, \pazocal{R}'/(t^m)\}.$$
Ideally, we would like to extend the definition of  $\omega_{m,r}$ to a map from $\Lambda ^{3}(\pazocal{R},\pazocal{R}',\chi)^{\times} $ to $\Omega^{1}_{\underline{\pazocal{R}}/k}.$ 
This can be done when $r=m+1$ but it is not true if  $m+1<r.$ 

However, it turns out that for us purposes, we do not need these 1-forms themselves but only their residues and we  can construct these residues independently of all the choices. Suppose that $\pazocal{S}/k_m$ is a smooth algebra of relative dimension 1, with $x$ a closed point and $\eta$ the generic point of its spectrum.    Suppose that $\pazocal{R},\, \pazocal{R}'/k_r$ are liftings of $\pazocal{S}_{\eta}$ to $k_r,$ with $\chi$ the corresponding isomorphism from $\pazocal{R}/(t^m)$ to $\pazocal{R}'/(t^m).$   We  construct a map 
$$
res_{x}\omega_{m,r}:\Lambda ^3 (\pazocal{R},\pazocal{R}',\chi)^{\times} \to k',
$$
where $k'$ is the residue field of $x,$ which is functorial and independent of all the choices. Let $
\tilde{\chi}:\pazocal{R} \to \pazocal{R}'
$
be an isomorphism of $k_r$-algebras which is a lifting of  $\chi.$ Choosing also an isomorphism  $\underline{\pazocal{R}}_r \simeq \pazocal{R}$ of $k_{r}$-algebras,  provides us with an identification
$$
\xymatrix{
(\pazocal{R},\pazocal{R}',\chi)^{\times}  \ar^{\tilde{\chi}^{*}}[r] & (\underline{\pazocal{R}}_r,(t^m))^{\times}}.
$$
Let  $\underline{\psi}$ denote the  isomorphism $\underline{\pazocal{R}}\to \underline{\pazocal{S}}_{\eta}$ induced by the one from $\pazocal{R}/(t^m)$ to $\pazocal{S}_{\eta}.$  Then we define $res_{x}\omega_{m,r}$ by the composition 
\begin{align*}
\xymatrix@R+1pc@C+1pc{
 \Lambda^3 (\pazocal{R},\pazocal{R}',\chi)^{\times} \ar^{\Lambda^3\tilde{\chi}^{*}}[r] & \Lambda^3(\underline{\pazocal{R}}_r,(t^m))^{\times} \ar[r]^{\;\;\;\omega_{m,r}} &\Omega^1 _{\underline{\pazocal{R}}/k} \ar[r]^{d\underline{\psi}}& \Omega^1 _{\underline{\pazocal{S}}_{\eta}/k} \ar^{res_x}[r]& k'.}
\end{align*}
We prove that this composition is independent of all the choices. This statement, together with the construction of the function, takes up the whole of \textsection 4 and 5. Applying this construction in the above context, we see that $\tilde{\gamma}_{i\eta}-\delta(\tilde{\beta}_{j_ci})$ and $\tilde{\gamma}_{j_c}$ are two liftings of the same object $\gamma_{j_c}$ to two different generic liftings of $\hat{\pazocal{O}}_{C,c}.$ Therefore, the expression  
$$
res_{c}\omega_{m,r}(\tilde{\gamma}_{i\eta}-\delta(\tilde{\beta}_{j_ci,\eta}),\tilde{\gamma}_{j_c})\in k(c)$$ is defined. 

Applying traces   and taking the sum over all the closed points, we obtain the expression in (\ref{firstformrho}). Next we show that the sum is in fact a finite sum. The above construction involves many choices and would  be completely useless if it depended on anything other than the initial data. This is the content of Theorem \ref{mainthm}, our main theorem. Because of its basic properties that we prove below, $\rho_{m,r}$ deserves to be called a regulator. 

In  Corollary \ref{cor inf dilog}, we specialize to triples of functions which  gives us the extension of the infinitesimal Chow dilogarithm of \cite{unv1} to   higher modulus. Just as in \cite{unv1}, this construction gives an analog of the strong reciprocity conjecture of Goncharov \cite{arak}. Finally, in the last section we obtain  invariants of cycles in $\underline{z}^{2}_{f}(k_{\infty},3)$ which satisfy a reciprocity property as in Theorem \ref{modulus theorem}. Namely,  if two cycles are the same modulo $(t^m)$ then their invariants $\rho_{m,r}$ are the same. This generalizes Park's theorem \cite{park}, which was proved in the context of additive Chow groups and for $r=m+1.$  After the category  of motives over rings with nilpotents is constructed, we expect these invariants to induce the regulators in this category. 

Finally, let us describe the differences of this paper with \cite{unv1}, where the case $m=2$ was handled.  In \cite{unv1}, since the only possible $r$ is 3 and hence satisfies $r=m+1,$ the map $\omega_{2,3}$ can be defined as a map from $\Lambda ^{3}(\pazocal{R},\pazocal{R}',\chi)^{\times} \to \Omega ^{1} _{\underline{\pazocal{R}}/k}$. This is not true in general and this is  why we have to pursue a different approach in this paper. Also since the formulas get quite complicated for a general $m,$ in this paper we follow a  more conceptual way of constructing $res_{c}\omega_{m,r}.$  We separate the construction into two parts as we described above. First we construct an absolute object $L_{m,r}$ which does not depend on $R$ being of dimension 1. This is done by an explicit computation. Next in order to define $\omega_{m,r}$ in the split case of $\pazocal{R}:=R_{r},$ and with $R/k$ smooth of relative dimension 1, we use  the computation of the Milnor $K$-theory of truncated polynomial rings in order to express an element $\alpha \in  im((1+t^mR_{r})^{\times } \otimes \Lambda^{2}R_{r}^{\times}) \subseteq \Lambda^{3}R_{r}^{\times})$ as a sum of an element of the form $\delta_{r}(\gamma),$ for some $\gamma \in B_{2}(R_{r})\otimes R_{r} ^{\times}$  and another element $\varepsilon \in   im( (1+t^mR_{r})^{\times } \otimes R_{r}^{\times}\otimes k^{\times})\subseteq \Lambda^{3}R_{r}^{\times}.$  The expression of $\varepsilon$ has too constant terms so that its image under $\omega_{m,r}$ should be 0. Therefore, we essentially use $L_{m,r}(\gamma)$ to unambiguously construct $\omega_{m,r}.$ After showing that  this construction is well-defined, we find an explicit description of it. Later in the non-split case we show that the residue of the 1-form can be unambiguously defined. The rest of the proof follows more or less along the same lines as that of \cite{unv1}, except that the details and the computations are more difficult.

We give an outline of the paper. In \textsection 2, we give a review of the construction in \cite{unv1} of the additive dilogarithm on the Bloch group of a truncated polynomial ring.  In \textsection 3, we describe the infinitesimal part of the Milnor $K$-theory of a local $\mathbb{Q}$-algebra endowed with a nilpotent ideal, which is split, in terms of K\"{a}hler differentials. Without any doubt the results in this section are known to the experts and we do not claim originality. The reason for our inclusion of this section is  first that we could not find an easily quotable statement in the full generality which we will need in our later work, and second that we found a short argument which is in line with the general set-up of this paper. 
In \textsection 4, for a regular local $\mathbb{Q}$-algebra $R,$ we define  regulators $B_{2}(R_{m})\otimes R_{m} ^{\times}$ to $\Omega^{1}_{R}$ for every $m<r<2m,$ which vanishes on boundaries. This construction depends on the splitting of $R_m$ in an essential way. In \textsection 5, we introduce the main object of this paper: for a smooth algebra $R$ of relative dimension 1 over $k,$ we define regulators $\omega_{m,r}:\Lambda ^{3}(R_r,(t^m))^{\times} \to \Omega^{1}_{R/k},$ for each  $m<r<2m.$  In \textsection 6, we compute the residues of the value of  $\omega_{m,r}$ on good liftings. In \textsection 7, we use the results of the previous sections to construct the regulator from ${\rm H}^3_{B}(C,\mathbb{Q}(3))$ and  specializing to  triples of rational functions we obtain the infinitesimal Chow dilogarithm of higher modulus. Finally, using the infinitesimal Chow dilogarithm, we construct an invariant of codimension 2 cycles in the 3 dimensional affine space over $k_m.$ 

{\bf Convention.} We are interested in everything modulo torsion. Therefore, {\it we tensor all abelian groups under consideration with $\mathbb{Q}$ without explicitly signifying this in the notation.} For example, $K_{n} ^{M}(A)$ denotes Milnor $K$-theory of $A$ tensored with $\mathbb{Q}$ etc. For an appropriate functor $F,$  we let $F(R_{\infty}) ^{\circ}:=ker (F(R_{\infty}) \to F(R))$ (resp. $F(R_{m}) ^{\circ}:=ker (F(R_{m}) \to F(R))$), denote the infinitesimal part of $F(R_{\infty})$ (resp. $F(R_{m})$).

\section{Additive dilogarithm of higher modulus}

In this section, we review and rephrase the theory of the additive dilogarithm over truncated polynomial rings in a manner which we will need in the remainder of the paper. Further results for this function can be found in \cite{unv1}. 

For a $\mathbb{Q}$-algebra $R,$ let 
 $R_{\infty}:=R[[t]],$ denote the formal power series over $R$  and $R_{m}:=R_{\infty}/(t^m)$ the truncated polynomial ring of modulus $m$ over $R.$ Since $R$ is a $\mathbb{Q}$-algebra we have the  logarithm $\log: (1+tR_{\infty})^{\times } \to R_{\infty}$ given by
 $\log (1+z):=\sum _{1\leq n} (-1)^{n+1}\frac{z^n}{n},$ for $z \in tR_{\infty}.$  Let $\log ^{\circ}: R_{\infty} ^{\times} \to R_{\infty},$ be the branch  of the logarithm associated to  the  splitting of $R_{\infty} \twoheadrightarrow R$ corresponding to the inclusion $R \hookrightarrow R_{\infty},$  defined as $\log ^{\circ}(\alpha):=\log (\frac{\alpha}{\alpha(0)}).$ If $q=\sum _{0\leq i} q_i t^i \in R_{\infty}$ and $1 \leq  a $  then let  $q|_{a}:=\sum _{0\leq i<a} q_i t^i 
  \in R_{\infty},$ denote the truncation of $q$ to the sum of the first $a$-terms, and $t_a(q):=q_a,$ the coefficient of $t^a$ in $q.$ If $u \in tR_{\infty}$ and $s(1-s) \in R^{\times},$ we let    
\begin{eqnarray}\label{adddilogformula}
\ell i_{m,r}(se^{u}):=t_{r-1}(\log^{\circ}(1-se^{u|_m}) \cdot \frac{\partial u}{\partial t}\big| _{r-m})  ,
\end{eqnarray}
for $m < r < 2m.$ Fixing $m\geq 2,$ these $\ell i_{m,r} $'s, for $m<r<2m$ together constitute a regulator for the infinitesimal part of the $K$-group $K_{3}(R_{m})^{(2)}$ exactly analogous to the Bloch-Wigner dilogarithm in the complex case \cite{blo}, \cite{sus}, \cite{unv1}. 

For any ring $A,$ we let $A^{\flat}:=\{a \in A|a(1-a) \in A^{\times} \}.$  Since every element of $R_{\infty} ^{\flat}$ can be written in the form $se^{u}$ as above, we can linearly extend $\ell i _{m,r}$, to obtain a map from the vector space $\mathbb{Q}[R_{\infty} ^{\flat}]$ with basis $R_{\infty} ^{\flat}.$ We denote this map by the same symbol. 

For a local $\mathbb{Q}$-algebra $R$, let $B_{2}(R)$ is the $\mathbb{Q}$-space generated by the symbols $[x],$ with $x(1-x) \in R^{\times},$ modulo the subspace generated by 
$$
 [x]-[y]+[y/x]-[(1-x^{-1})/(1-y^{-1})]+[(1-x)/(1-y)],
 $$  
 for all $x,y \in R^{\times} $ such that $(1-x)(1-y)(1-x/y) \in R^{\times}.$
The Bloch complex $\delta:B_{2}(R) \to \Lambda^2 R^{\times} ,$ such that $\delta([x]):=(1-x)\wedge x,$ computes the weight 2 motivic cohomology of $R,$ when $R$ is a field. When we would like to specify the $\delta$ defined on $B_{2}(R_{\infty})$ (resp. $B_{2}(R_{m})$) we denote it by $\delta_{\infty}$ (resp. $\delta_{m}$).

 Let $V$ be a free $R$ module with basis $\{e_{i} \}_{i \in I}$ and $\{ e_{i} ^{\vee} \} _{i \in  I}$ the dual basis of $V ^{\vee}.$ Given $v$ and  $\alpha=\sum_{i \in I} a_{i} e_{i}$ in $V,$ we let 
$$
(v|\alpha):=\sum _{i \in I} a_{i}e_{i} ^{\vee} (v) \in R.
 $$
If there is an ordering on $I,$ we let $\{e_{i} \wedge e_{j}\}_{i>j}$ be the corresponding basis of $\Lambda^{2}  V.$ Then, with the above notation, the expression $(w|\beta),$ for $w,\,\beta \in \Lambda ^2 V,$ is defined.  We consider $tR_{\infty},$ as a free $R$-module with basis $\{t^i \}_{1 \leq i }.$

Let us denote the composition of  $ B_{2}(R_{\infty}) \xrightarrow{\delta} \Lambda ^{2} R_{\infty} ^{\times} $ with the canonical projection $\mathbb{Q}[R_{\infty} ^{\flat}] \to B_{2}(R_{\infty})$ also by $\delta.$ 
Also denote the  map 
$$
\Lambda ^{2} R_{\infty} ^{\times}  \to \Lambda ^{2} tR_{\infty}  \twoheadrightarrow \Lambda ^{2}_ {R} tR_{\infty}
$$
induced by $\Lambda^{2} \log ^{\circ}: \Lambda ^{2}R_{\infty} ^{\times} \to \Lambda ^{2}tR_{\infty},$ by the same symbol.  

\begin{proposition}\label{prop-alter-li}
With the notation above, for $\alpha \in \mathbb{Q}[R_{\infty} ^{\flat}]$ and $2 \leq m<r<2m,$ we have 
\begin{eqnarray}\label{twodilogs}
 \ell i_{m,r}(\alpha)=\Big(\Lambda^{2} \log ^{\circ} (\delta(\alpha))| \sum _{1\leq i \leq r-m} it^{r-i}\wedge t^i\Big),
\end{eqnarray}
and this function descends through the canonical projections 
$$
\mathbb{Q}[R_{\infty} ^{\flat}] \to B_{2}(R_{\infty}) \to B_{2}(R_{m}),
$$
to define a map from $B_{2}(R_{m})$ to $R,$ denoted by the same notation. 
\end{proposition}

\begin{proof}
We proved in \cite[Prop. 2.2.1]{unv1} that the function defined by the right hand side of (\ref{twodilogs}), temporarily denote it by $\ell i_{m,r} ^{*},$ descends to give a map from $\mathbb{Q}[R_{m} ^{\flat}]$ and in  \cite[Prop. 2.2.2]{unv1} that it descends to give a map from $B_{2}(R_{m}).$ Therefore it only remains to prove the equality (\ref{twodilogs}).

With the notation $\ell_i(\alpha):=t_i(\log^{\circ}(\alpha)),$  $\ell i_{m,r} ^{*}$ can be rewritten as 
$$\ell i_{m,r} ^{*}=\Big( \sum_{1\leq i \leq r-m}i \cdot \ell_{r-i}\wedge \ell _i \Big) \circ \delta.$$ 
Then we have $\ell i_{m,r} ^{*}(se^{u})=\ell i_{m,r} ^{*}(s e^{u|_{m}}),$ since we know that $\ell i_{m,r} ^{*}$ descends to $\mathbb{Q}[R_{m} ^{\flat}].$ We have $\ell _{i} (se^{u|_m})=u_{i},$ for $1\leq i <m$ and $\ell _{i} (se^{u|_m})=0,$ for $m \leq i.$ Using this we obtain that  $\ell i_{m,r} ^{*}(s e^{u|_{m}})=\sum_{1\leq i \leq r-m}i \cdot \ell_{r-i}(1-se^{u|_{m}}) \cdot  u_i=\ell i_{m,r} (se^u).$ 
\end{proof}
 Let us give a name to the essential map  which constitute $\ell i_{m,r}.$ 
 \begin{definition}\label{defnlmr}
We  denote the map from $\Lambda^2 R_{\infty} ^{\times}$ to $R$ which sends $\alpha \wedge \beta $ to 
 $$(\Lambda ^2 \log ^{\circ} (\alpha\wedge \beta)| \sum_{1 \leq i \leq r-m} it^{r-i} \wedge t^i)$$
 by $\ell_{m,r}.$ It is clear  that $\ell_{m,r}: \Lambda^2 R_{\infty} ^{\times} \to R $ factors through the projection $\Lambda^2 R_{\infty} ^{\times}  \to \Lambda ^2 R_{r} ^{\times}.$ The additive dilogarithm above is given in terms of this function as 
 $$
 \ell i_{m,r}=\ell_{m,r}\circ\delta_{\infty}=\ell_{m,r}\circ\delta_{r}.
 $$
 \end{definition}

We will use the main result from \cite{unv1}, there it was stated in the case when $R$ is a field of characteristic 0, but the same proof works when $R$ is a regular, local $\mathbb{Q}$-algebra.

\begin{theorem}\label{theorem-b2}
The complex  $B_{2}(R_{m}) ^{\circ} \xrightarrow{\delta ^{\circ}} (\Lambda^2 R_{m} ^{\times})^{\circ}$ computes the infinitesimal part of the weight two motivic cohomology of $R_{m},$ and the map $\oplus _{m <r<2m} \ell i_{m,r}$ induces an isomorphism 
$$
{\rm HC}_{2} ^{\circ}(R_{m})^{(1)}  \simeq K_{3} ^{\circ}(R_m)^{(2)}  \simeq ker (\delta ^{\circ}) \xrightarrow{\sim} R^{\oplus (m-1)}
$$
from the relative cyclic homology group ${\rm HC}_{2} ^{\circ}(R_{m})^{(1)} $ to $R^{\oplus
(m-1)}.$ 
\end{theorem}

\section{Infinitesimal Milnor $K$-theory of local rings} Suppose that $R$ is a local $\mathbb{Q}$-algebra and $A$ is an  $R$-algebra, together with a nilpotent ideal $I$ such that  the natural map $R \to A/I$ is an isomorphism. Then the Milnor $K$-theory $K_{n} ^{M}(A)$ of $A,$ naturally splits into a direct sum $K_{n} ^{M}(A)=K_{n} ^{M}(R) \oplus K_{n} ^{M}(A) ^{\circ}.$ In this section, we will describe this infinitesimal part $K_{n} ^{M}(A) ^{\circ}$ in terms of K\"{a}hler differentials. It is easy to find such an isomorphism using Goodwillie's theorem \cite{good}, and standard computations in cyclic homology. However, in the next section, we need an explicit description of this isomorphism  in order to determine which symbols vanish in the corresponding Milnor $K$-group. Fortunately, determining what this isomorphism turns out to be quite easy.  By the functoriality and the multiplicativity of the isomorphism, we reduce the computation to the case of $K_2^{M}$ of the  dual numbers over $R$ where the computation is easy.

There is no doubt that the results in this section are  well-known and we do not claim any originality. We simply have not been able to find a description of the map $\varphi$ below which is easily quotable in the literature. Since our discussion is quite short we did not refrain from including it in the present paper. We will only need the result below for $A=R_m.$ On the other hand, in a future work we will need this result in full generality which justifies our somewhat more general discussion:

\begin{proposition}\label{milnor kahler}
There exists a unique map $\varphi:K_{n} ^{M}(A)^{\circ}\to \Omega^{n-1}_{A}/(d\Omega^{n-2}_{A}+\Omega^{n-1}_{R})$ such that 
\begin{eqnarray}\label{formulaphi}
\varphi (\{\alpha,\beta_1,\cdots, \beta_{n-1} \})=\log (\alpha)\frac{d\beta_1}{\beta_1}\wedge \cdots \wedge \frac{d\beta_{n-1}}{\beta_{n-1}},
\end{eqnarray}
for $\alpha \in 1+I $ and $\beta_{1}, \cdots, \beta_{n-1} \in A^{\times},$
 and this map is an isomorphism. 
\end{proposition}
\begin{proof}
The uniqueness follows since the infinitesimal part of Milnor $K$-theory is generated by terms $\{\alpha,\beta_1,\cdots, \beta_{n-1} \}$ as in the statement. 

We  define a functorial map $\varphi$ by the following composition: 
\begin{align}\label{injection}
K_{n} ^{M}(A)^{\circ}   \to  K_{n}^{(n)} (A)^{\circ } \xrightarrow{\sim} {\rm HC}_{n-1} ^{(n-1)}(A)^{\circ}=(\Omega^{n-1} _{A}/d\Omega^{n-2} _{A})^{\circ}=\Omega^{n-1} _A/(d\Omega^{n-2}_A+\Omega^{n-1}_R).
\end{align}
The first map is the multiplicative map induced by the isomorphism when $n=1,$ the second one is the Goodwillie isomorphism \cite{good}, and the last one is given by \cite[Theorem 4.6.8]{lod}.

By  Nesterenko-Suslin's theorem  \cite{ns}, Milnor $K$-theory is  the first obstruction to the stability of the   homology of general linear groups:
$$
K_{n} ^{M}(A) \simeq {\rm H}_{n}({\rm GL}_{n}(A),\mathbb{Q})/{\rm H}_{n}({\rm GL}_{n-1}(A),\mathbb{Q}).
$$
Moreover, the composition 
$$
K_{n} ^{M}(A)^{\circ}  \to  K_{n}^{(n)} (A)^{\circ }  \to Prim({\rm H}_{n} ({\rm GL}(A),\mathbb{Q})) \to {\rm H}_{n} ({\rm GL}(A),\mathbb{Q})\simeq  {\rm H}_{n} ({\rm GL}_{n}(A),\mathbb{Q}) \twoheadrightarrow K_{n} ^M (A)
$$
is multiplication by $(n-1)!$ \cite{ns}. This implies the injectivity of $\varphi.$ It only remains to prove the property (\ref{formulaphi}), since then the surjectivity of $\varphi$ also follows. 

The multiplicativity of $\varphi$ takes the following form:  for $a, \, b \in K_{m} ^M (A) ^{\circ},$  
$
\varphi(a\cdot b)=\varphi (a) \wedge d(\varphi(b)).
$
We do induction on $n.$ The statement is clear for $n=1.$  We show that we may assume that $\beta_ i \in R^{\times}$:   

\begin{lemma}
Suppose that we have the formula (\ref{formulaphi}) for $\alpha \in  1+ I$ and $\beta_{i} \in R^{\times},$ for $1\leq i \leq n-1,$ then we have the same formula for  $\alpha \in  1+I$ and $\beta_{i} \in A^{\times},$ for $1\leq i \leq n-1.$
\end{lemma}

\begin{proof}
We do induction on the number of $\beta_i$  which are not in $R^{\times}.$ If all of them are in $R^{\times},$ the hypothesis of the lemma gives the expression. If there is at least one $\beta_{i}$ which is not in $R^{\times},$ without loss of generality  assume that  $\beta_{n-1} \notin R^{\times}. $ Let us write $\beta_{n-1}:=\lambda \cdot \beta,$ with $\lambda \in R^{\times}$ and $\beta \in 1+I.$ Then 
\begin{eqnarray*}
\varphi (\{\alpha,\beta_1,\cdots, \beta_{n-1} \})=\varphi (\{\alpha,\beta_1,\cdots,\beta_{n-2}, \lambda \})+\varphi (\{\alpha,\beta_1,\cdots,\beta_{n-2}, \beta \}).
\end{eqnarray*}
By the multiplicativity of $\varphi,$ the formula for $n=1,$ and the induction hypothesis on $n$, we have 
$$
\varphi (\{\alpha,\beta_1,\cdots,\beta_{n-2}, \beta \})=\varphi (\{\alpha,\beta_1,\cdots,\beta_{n-2} \}) \wedge d(\log (\beta))=\log(\alpha) \frac{d \beta_1}{\beta_1}\wedge \cdots  \frac{d \beta_{n-2}}{\beta_{n-2}}\wedge \frac{d \beta}{\beta}.
$$
By the induction hypothesis on the number of $\beta_{i}$ not in $R^{\times},$   
we have  
$$
\varphi (\{\alpha,\beta_1,\cdots,\beta_{m-1}, \lambda \})=\log (\alpha) \frac{d\beta_1}{\beta_1}\wedge \cdots \wedge \frac{d\beta_{m-1}}{\beta_{m-1}}\wedge \frac{d \lambda}{\lambda}.
$$
Adding these two expressions, we obtain the expression we were looking for. 
\end{proof}

The above lemma shows that we may without loss of generality assume that the $\beta_i \in R^{\times}.$ The next lemma shows that we may  also  assume that $A=R_{r}$ and $\alpha=1+t.$  

\begin{lemma}
Suppose that we have the formula (\ref{formulaphi}) for $\alpha=1+t$ and $\beta _i \in R^{\times}$ for $1 \leq i \leq n-1,$ for the ring $R_{r}:=R[t]/(t^r).$ Then we have the same formula for any $A$ as above.
\end{lemma}

\begin{proof}
Given $\alpha \in 1+I \subseteq A^{\times}$ and $\beta_i \in R^{\times}.$ Since  $\alpha-1$ is nilpotent,   we have an $R$-algebra morphism $\psi:R_{r} \to A,$ for some $r,$ such that $\psi(t)=\alpha-1.$ The result then follows by the functoriality of $\varphi$ since the map induced by $\psi$ maps $\{1+t,\beta_1,\cdots, \beta_{n-1} \}$ to $\{\alpha,\beta_1,\cdots, \beta_{n-1} \}.$  
\end{proof}
Next we show that we can also assume that $r=2.$ 
\begin{lemma}
Suppose that we have the formula (\ref{formulaphi}) for $\alpha=1+t$ and $\beta _i \in R^{\times}$ for $1 \leq i \leq n-1,$ for the ring $R_{2}.$ Then we have the same formula for any $A$ as above. 
\end{lemma}

\begin{proof}
We need to prove the result for $1+t \in R_{r},$ and $\beta_i \in R^{\times}.$ Since 
$1+t=e^{\log (1+t)},$ it is a product of elements of the form $e^{at^m},$ for $1 \leq m <r,$ and $a\in \mathbb{Q}.$ Therefore it is enough to prove the formula for elements as above  with  $\alpha=e^{at^m}.$  

Since the element $\{e^{at^m},\beta_1,\cdots,\beta_{n-1} \}$ is of $\star$-weight $m,$ its image under $\varphi$ is  in $(\Omega^{n-1}_{R_r}/d\Omega^{n-2}_{R_r})^{[m]},$ the $\star$-weight $m$ part of $\Omega^{n-1}_{R_r}/d\Omega^{n-2}_{R_r}.$ On the other hand the natural surjection $R_r \to R_{m+1}$ induces an isomorphism 
$$
(\Omega^{n-1}_{R_r}/d\Omega^{n-2}_{R_r})^{[m]} \simeq (\Omega^{n-1}_{R_{m+1}}/d\Omega^{n-2}_{R_{m+1}})^{[m]}.
$$
 Therefore, without loss of generality, we will assume that $r=m+1.$ Then we use the map from $R_2$ to $R_{m+1}$ that sends $t$ to $at^m.$ This map sends $1+t$ to $e^{at^m}$ and  hence maps $\{1+t,\beta_1,\cdots, \beta_{n-1} \}$ to $\{e^{at^m},\beta_1,\cdots, \beta_{n-1} \}.$ Therefore, again by the functoriality of $\varphi,$ the result follows from the assumption on $R_2.$ 
 \end{proof}
 To finish the proof, we will need a special identity in $K_{2} ^{M}(R_3):$ 
 \begin{lemma}
 We have the following relation in $K_{2} ^{M} (R_3):$
 $$
 2\{1+\frac{t^2}{2},\lambda \}=\{1+\frac{t}{\lambda},1+\lambda t \},
 $$
 for any $\lambda \in R^{\times}.$ 
 \end{lemma}

\begin{proof} It is possible to give a direct computational proof of this statement. We choose to give a proof which is based on the ideas in this section. 

First suppose that $R$ is a field. We know that both sides are in $K_{2} ^{M} (R_3) ^{\circ}.$ We  know from \cite{gra} that the map  $K_{2} ^{M} (R_3) ^{\circ} \to (\Omega^{1} _{R_3}/dR_{3})^{\circ}$ which sends $\{\alpha,\beta \}$ to $\log (\alpha) \frac{d\beta}{\beta},$ where $\alpha -1 \in (t),$ is an isomorphism.  

The left hand side goes to $t^2 \frac{d \lambda}{\lambda},$ whereas the right hand side goes to 
$$
\frac{t}{\lambda} d(\lambda t)=t^2\frac{d\lambda}{\lambda}+tdt=t^2\frac{d\lambda}{\lambda}+\frac{1}{2}dt^2=t^2\frac{d\lambda}{\lambda}
$$ 
in $(\Omega^{1} _{R_3}/dR_{3})^{\circ}.$ 
This proves the statement when $R$ is a field. 

In general, the statement for $\mathbb{Q}[x,x^{-1}]$ implies the one for a general $R$ by sending $x$ to $\lambda.$ Finally, if we can show that $K_{2} ^{M}(\mathbb{Q}[x,x^{-1}]_{3})^{\circ} \to K_{2} ^{M}(\mathbb{Q}(x)_{3})^{\circ} $ is an injection, the known statement for $\mathbb{Q}(x)$ implies the one for $\mathbb{Q}[x,x^{-1}].$ This injectivity follows from the commutative diagram 
$$
\xymatrix{
 K_{2} ^{M}(\mathbb{Q}[x,x^{-1}]_3)^{\circ} \ar[d] \ar@{^{(}->}^{\varphi \;\;\;\;\;\;\;\;\;\;\;\;}[r] & (\Omega^{1}_{\mathbb{Q}[x,x^{-1}]_3}/d(\mathbb{Q}[x,x^{-1}]_3))^{\circ} \ar[d]  & t\Omega^{1}_{\mathbb{Q}[x,x^{-1}]}\oplus t^2 \Omega^{1}_{\mathbb{Q}[x,x^{-1}]} \ar^{\;\;\;\;\;\sim}[l]  \ar@{^{(}->}[d]\\
 K_{2} ^{M}(\mathbb{Q}(x)_3)^{\circ} \ar@{^{(}->}^{\varphi \;\;\;\;\;\;\;\;\;\;\;\;}[r] & (\Omega^{1}_{\mathbb{Q}(x)_3}/d(\mathbb{Q}(x)_3))^{\circ}  & t\Omega^{1}_{\mathbb{Q}(x)}\oplus t^2 \Omega^{1}_{\mathbb{Q}(x)},  \ar^{\;\;\;\;\;\;\;\sim}[l]}
$$
where the injectivity of $\varphi$ was proven above. This finishes the proof of the lemma. 
\end{proof}

Finally, we prove the result for $R_{2}.$ 

\begin{proposition}
Let $\alpha=1+t \in R_{2} ^{\times}$ and $\beta_i \in R^{\times},$ for $1 \leq i \leq n-1,$ then 
$
\varphi(\{\alpha,\beta_1,\cdots,\beta_{n-1})
$
is given by (\ref{formulaphi}).
\end{proposition}

\begin{proof}
Note the map $\psi$ from $R_2$ to $R_3$ that sends $t$ to $\frac{t^2}{2}.$ This map induces an isomorphism 
$$
(\Omega^{n-1} _{R_2}/d  \Omega^{n-2} _{R_2}) ^{[1]} \simeq (\Omega^{n-1} _{R_3}/d  \Omega^{n-2} _{R_3}) ^{[2]}.
$$
Therefore we only need to compute the image of 
\begin{eqnarray}\label{k_3exp}
\{1+\frac{t^2}{2},\beta_1,\cdots, \beta_{n-1} \}
\end{eqnarray}
in $\Omega^{n-1} _{R_3}/d  \Omega^{n-2} _{R_3}) ^{[2]}.$ 
By the previous lemma, we know that 
$$
 \{1+\frac{t^2}{2},\beta_1 \}=\frac{1}{2}\{1+\frac{t}{\beta_1},1+\beta_1 t \},
 $$
which implies that (\ref{k_3exp}) is equal to 
\begin{eqnarray}\label{redexp}
\frac{1}{2}\{1+\frac{t}{\beta_1},1+\beta_1 t,\beta_2,\cdots , \beta_{n-1} \}.
\end{eqnarray}
This last expression is the $\frac{1}{2}$ times the product of $\{1+\frac{t}{\beta_1} \} \in K_{1} ^{M}(R_{3})^{\circ}$ and $$\{1+\beta_1 t,\beta_2,\cdots , \beta_{n-1} \} \in K_{n-1} ^{M}(R_{3})^{\circ} .$$

By the induction hypothesis on $n,$  $$\varphi(\{1+\beta_1 t,\beta_2,\cdots , \beta_{n-1} \} )=
\log (1+\beta_1t) \frac{d\beta_2}{\beta_2} \wedge \cdots \wedge\frac{d \beta_{n-1}}{\beta_{n-1}}.
$$
Since  $\varphi$ is multiplicative, this implies that   (\ref{redexp}) is sent by $\varphi$ to 

$$
\frac{1}{2}\log (1+\frac{t}{\beta_1} )d\log (1+\beta_1t) \frac{d\beta_2}{\beta_2} \wedge \cdots \wedge\frac{d \beta_{n-1}}{\beta_{n-1}}=\log (1+\frac{t^2}{2})\frac{d\beta_1}{\beta_1} \wedge \cdots \wedge \frac{d\beta_{n-1}}{\beta_{n-1}}.
$$
\end{proof}
This finishes the proof of Proposition \ref{milnor kahler}.
\end{proof}

In the case of truncated polynomial rings, we can also describe this isomorphism as follows: 
\begin{corollary}\label{cor milnor}
The map $\lambda_{i}: \Lambda ^{n}R_{\infty} ^{\times} \to \Omega^{n-1}_{R}$  given by 
$$
\lambda_{i}(a_{1}\wedge \cdots \wedge a_{n})=res_{t=0}\frac{1}{t^i}d\log (a_1)\wedge \cdots \wedge d\log(a_n)\in \Omega^{n-1}_{R},$$
for $1\leq i<r,$ descends to give a map $K_{n} ^{M}(R_r)^{\circ} \to \Omega^{n-1}_{R}.$ Their sums induce an isomorphism:
\begin{align*}
K_{n} ^{M}(R_r)^{\circ}\to \oplus _{1\leq i <r}\Omega^{n-1}_{R}.    
\end{align*}
\end{corollary}

\begin{proof}
For $1\leq i <r,$ we let  $\mu_i:(\Omega^{n-1}_{R_r}/d(\Omega^{n-2}_{R_r}))^{\circ} \to \Omega^{n-1}_{R}$ be  given by 
$\mu_i(w):=res_{t=0}\frac{1}{t^i}d\omega.$ The induced map 
$$
(\Omega^{n-1}_{R_r}/d(\Omega^{n-2}_{R_r}))^{\circ} \to \oplus _{1\leq i <r}\Omega^{n-1}_{R}
$$
is an isomorphism. The corollary then follows from Proposition \ref{milnor kahler}.
\end{proof}

\section{Construction of maps from $B_{2}(R_m) \otimes R_{m} ^{\times}$ to $\Omega^{1}_{R}$}

In this section, we will deal with the absolute case, i.e. we will not assume the existence of a basefield $k$ over which $R$ is smooth of relative dimension 1.  We will  only assume that $R$ is a regular, local $\mathbb{Q}$-algebra. Note that we have a complex 

\begin{align}\label{weight-3-seq}
   \mathbb{Q}[R_{m} ^{\flat}] \to B_{2}(R_{m})\otimes R_{m} ^{\times} \to \Lambda ^{3}R_{m} ^{\times}, 
\end{align}
where the first map sends $[x]$ to $[x] \otimes x$ and the second one sends $[x] \otimes y$ to $\delta(x) \wedge y.$ Abusing the notation, we will  denote all the differentials in this complex by $\delta.$  The first group when divided by the appropriate relations is generally denoted by $B_{3}(R_{m})$ and the corresponding sequence sequence obtained is a candidate for the weight 3 motivic cohomology complex generalizing the Bloch complex of weight 2. This complex and its variants are defined and studied in detail in \cite{config}. Over the dual numbers of a field, this complex and its higher weight analogs, still called the Bloch complexes, were used in \cite{unv1.5} to construct the additive polylogarithms. 

Let us put the above cohomological complex in degrees $[1,3]$ and denote its cohomology after tensored with $\mathbb{Q}$ as ${\rm H}^{i}(R_{m},\mathbb{Q}(3))$ in degrees $i=2$ and 3.

\subsection{Preliminaries on the construction}

In this section, we fix $m $ and $r$ such that  $2\leq m<r<2m .$ We let $f(s,u):=\log ^{\circ}(1-se^u)=\log(\frac{1-se^u}{1-s}).$ As in the proof of Proposition \ref{prop-alter-li}, we define $\ell _{i}: R_{\infty}^{\times} \to R,$ by the formula   $\ell_i(a):=t_i(\log^{\circ}(a).$   Let us consider  the expression 
\begin{eqnarray}\label{indepterm}
\alpha_j:=\sum _{1 \leq i \leq j-1} i d\ell_{j-i} \wedge \ell_{i}=\sum _{a+b= j \atop {1 \leq a, \,b}}b d \ell _a \wedge \ell _b ,
\end{eqnarray}
for $m\leq j <r,$ which defines a map from $\Lambda^2 R_{\infty} ^{\times}$ to $\Omega^1 _{R}.$   We will use this expression to define a map from $B_{2}(R_{m})\otimes R_{m} ^{\times}. $ 

\begin{lemma}\label{alphalemma}
For $s \in R ^{\flat},$ and $u:=\sum _{0<i}u_it^i \in tR_{\infty},$  the expression $\alpha_{j}(\delta(se^u))$ does not contain a $du_i$ term. In fact, letting $f_{s}:=\frac{\partial f}{\partial s}$ and $u_{t}:=\frac{\partial u}{\partial t}=\sum _{0<i}iu_it^{i-1},$
$$
\alpha_{j}(\delta(se^u))=t_{j-1}(  f_{s} u_t)ds=t_{j-1}(  \frac{\partial \log ^{\circ}(1-se^u)}{\partial s}\cdot  \frac{\partial u}{\partial t})ds.
$$
\end{lemma}

\begin{proof} Let us write $f(s,u)=:f=\sum _{0<i} f_it^i.$
The expression $i d\ell _{j-i} \wedge \ell _{i} $ evaluated on $\delta(se^{u})$ is equal to 

  $$
  id(f_{j-i})u_i-if_idu_{j-i}=  id(f_{j-i})u_i+(j-i)f_idu_{j-i}-jf_idu_{j-i}.
  $$
Summing these, we find that  
 $$
\alpha_{j}(\delta(se^u))= \sum _{1\leq i\leq j-1}(id(f_{j-i})u_i+(j-i)f_idu_{j-i})-j \sum _{1\leq i\leq j-1}f_idu_{j-i}.
 $$
Let $Du:=\sum _{1\leq i}du_i t^i$ and $u_t:=\frac{\partial u}{\partial t}.$ Then the last  expression can be  rewritten as 
 \begin{eqnarray}\label{eqnforD}
\;\;\;\;\;\;\;\;t_{j-1}(D(f u_t))-jt_j(f Du)=
 t_{j-1}(D(f u_t)-(f Du)_t)
 =
 t_{j-1}( Df u_t-f_t Du).
\end{eqnarray}

We would like to see that the coefficient of $du_{i}$  in (\ref{eqnforD}) is equal to $0.$ The coefficient of $du_i$ in $Df=D \log (\frac{1-se^u}{1-s})$ is equal to 
$\frac{-se^u}{1-se^u}t^i.$ Therefore, the coefficient of $du_i$ in (\ref{eqnforD})  is 
$$
t_{j-1-i} \Big( \frac{-se^u}{1-se^u}u_t-f_t\Big).
$$
Since  $f_t=\frac{\partial }{\partial t}(\log (1-se^u))=\frac{-se^u}{1-se^u}u_t,$ the last expression is 0. 

Therefore $\alpha_{j}(\delta(se^u))$ does not depend on the $du_{i}$'s, and can rewrite (\ref{eqnforD}) as
$$   
\alpha_{j}(\delta(se^u))= t_{j-1}( Df u_t-f_tDu)= t_{j-1}( f_s u_t)ds,
$$
where $f_{s}=\frac{\partial f}{\partial s}.$
\end{proof}

\begin{lemma}\label{lemmafs}
If $u=u|_{m}$ and $m \leq j <r,$  we have 
$$
jt_j(f)=st_{j-1}(f_s u_t).
$$
 \end{lemma}

\begin{proof} The expression $jt_{j}(f)-st_{j-1}(f_s u_t)$ is equal to 
 $$
 t_{j-1}(f_t -s(f_s u_t) )=
 t_{j-1}(\frac{\partial}{\partial t}\log(1-se^u) -s\frac{\partial}{\partial s}\log(\frac{1-se^u}{1-s}) \cdot u_t) .
$$
 Since 
 $$
 \frac{\partial}{\partial t}\log(1-se^u)=\frac{-se^u}{1-se^u}\cdot u_t=s\frac{\partial}{\partial s}\log(1-se^u)\cdot u_t,
 $$
 the above  expression is equal to 
 $
 t_{j-1}(\frac{s}{s-1}u_t),
 $
which is 0, under  the assumption that $u=u_1t+\cdots +u_{m-1}t^{m-1}$ and $m \leq j.$ 
  \end{proof}
Let $d \ell_{0}: R_{\infty} ^{\times} \to \Omega^{1} _{R}$ be defined as $d \ell _{0}(\alpha):= d \log (\alpha (0)).$ Note that $\ell _0$ itself is not defined, even though $\ell _{i}$ are defined for $i >0.$ There is an action of  $R^{\times}$ on the $R$-algebra $R_{\infty}$ by scaling the parameter $t.$ More precisely, for $\lambda \in R^{\times}$ the corresponding automorphism is given by sending $t$ to $\lambda t.$ We call this the $\star$-action following \cite{blo-es}.  For any functor $F$ from $R$-algebras to $\mathbb{Q}$-spaces (resp. $R$-modules), $F(R_{\infty})$ becomes a $\mathbb{Q}[R^{\times}]$-module  (resp. $R[R^{\times}]$-module) via the $\star$-action. For $w \in \mathbb{Z},$ an element $m \in M$ in a $\mathbb{Q}[R^{\times}]$-module  (resp. $R[R^{\times}]$-module) $M$ is said to be of $\star$-weight $w,$ if every $\lambda$ in $\mathbb{Q}^{\times}$ (resp. in $R^{\times}$) acts on $m$ as $\lambda \star m=\lambda ^{w}m.$ 
  
\begin{proposition}\label{wrongmap} The map $M_{m,r}$ defined as 
$$
M_{m,r}:=\ell i_{m,r}  \otimes d\ell_{0} -  \sum_{m\leq j <r }  \frac{r-j}{j}(\alpha_{j}\circ \delta) \otimes \ell_{r-j}
$$
gives a map from $B_{2}  (R_\infty)\otimes R_\infty ^{\times}$ to $ \Omega^{1} _{R},$ of $\star$-weight $r,$ which vanishes on the image under $\delta$ of  those $[se^ u]  \in \mathbb{Q}[R_\infty ^{\flat}],$  with $u=u|_m.$ In general, $M_{m,r}$  does not descend to a map from $B_{2}  (R_m)\otimes R_m ^{\times}.$ It does, however, if $r=m+1. $ 
\end{proposition}

\begin{proof} That $M_{m,r}$ is of $\star$-weight $r$ follows immediately from the expression for $\alpha _{j}(\delta (s e^u))$ in Lemma \ref{alphalemma}, which shows that $\alpha _{j}(\delta (s e^u))$ is of $\star$-weight $j.$ 

Let us now show that $M_{m,r}$ evaluated on $[se^{u}] \otimes  se^{u}  ,$ with $u=u|_m,$ is equal to 0.  By Lemma \ref{alphalemma}, $\alpha_j(\delta(se^u))$ is equal to $t_{j-1}(f_su_t)ds.$  This implies that  
$
\sum_{m\leq j <r }  \frac{r-j}{j}(\alpha_{j}\circ \delta) \otimes \ell_{r-j}
$
evaluated on  $[se^u] \otimes se^u$ is equal to 
$$
\sum_{m\leq j <r }  \frac{r-j}{j}t_{j-1}(f_su_t)u_{r-j} ds=\frac{1}{s}\sum_{m\leq j <r }  (r-j)t_{j}(f)u_{r-j} ds,
$$
by Lemma \ref{lemmafs}, since $u=u|_m.$ The final expression can be rewritten as
$$
t_{r-1}(f\cdot u_{t}|_{r-m})  \frac{ds}{s}=\ell i_{m,r} (se^u)\frac{ds}{s}=(\ell i_{m,r}\otimes d \ell_0) ([se^u]  \otimes se^u)
$$  since $u=u|_m.$  This proves the first part of the proposition.

When $r=m+1,$ $M_{m,r}$ takes the form 
 $$
 \ell i_{m,m+1}   \otimes d \ell_{0} -\frac{1}{m}\big((d \ell _{m-1}\wedge \ell _{1}+2 d\ell_{m-2} \wedge \ell _{2} + \cdots + (m-1) d\ell_{1} \wedge \ell_{m-1})\circ \delta\big)\otimes \ell _{1}.
 $$
Since all the functions in this expression depend on the classes of the elements in $R_{m},$ the statement easily follows. 
\end{proof}

\subsection{The regulator maps from ${\rm H}^{2}(R_{m},\mathbb{Q}(3))$ to $\Omega^{1} _{R}$} 
We would like to define maps 
$$
L_{m,r}: {\rm H}^{2}(R_{m},\mathbb{Q}(3)) \to \Omega^{1} _{R}
$$
based on the maps $M_{m,r}$ in Proposition \ref{wrongmap}. The problem with $M_{m,r}$ is that it does not descend to a map on  $B_{2}(R_{m})\otimes R_{m} ^{\times}$ in general.  We will modify $M_{m,r}$ slightly to correct this defect but keep the other properties to obtain $L_{m,r}.$ 

In order to simplify the notation from now on we are going to let $\ell i_{m,m}:=0.$ Note that $\ell i_{m,r}$ was previously defined only when $m+1\leq r \leq 2m-1$ so this will not cause any confusion. We define $\beta_{m}(j),$ for $m \leq j <2m-1,$  by 
\begin{eqnarray*}
\beta_{m}(j)&:=& d \ell i_{m,j} +   \sum _{ a+b=j \atop {1 \leq a, b<m}} b (d \ell _{a} \wedge \ell _{b})\circ \delta=\sum_{a+b=j \atop {m \leq a, \, 1 \leq b}}b d( \ell _a \wedge \ell_b)\circ \delta+   \sum _{ a+b=j \atop {1 \leq a, b<m}} b (d \ell _{a} \wedge \ell _{b})\circ \delta\\
& =& d \ell i_{m,j}+\alpha_j\circ \delta-\sum _{1\leq a \leq j-m}\big ((j-a) d \ell _a \wedge \ell _{j-a}+a d \ell _{j-a} \wedge \ell _{a} \big)\circ \delta
\end{eqnarray*}
and 
\begin{eqnarray}\label{eqn lmr}
L_{m,r}:=\ell i _{m,r} \otimes d\ell _0- \sum _{m\leq j <r } \big(\frac{r-j}{j}\beta_{m}(j) \otimes \ell _{r-j}- \ell i _{m,j} \otimes d \ell _{r-j} \big).
\end{eqnarray}
We would like to emphasize that, because of our conventions,  the summand that corresponds to $j=m $ is equal to $\frac{r-m}{m}\alpha _m \otimes \ell _{r-m}$ exactly as in the case of $M_{m,r},$ the terms corresponding to $m<j$ are  modified however.

\begin{lemma}\label{lemma-descent-to-m}
With the above definition, $L_{m,r}$ defines a map from $B_{2} (R_m) \otimes R_m ^{\times}$  to $\Omega^{1} _{R}$ of $\star$-weight $r.$
\end{lemma}

\begin{proof}
Since all the terms in the definition of $L_{m,r}$ depend on the variables modulo $t^m, $ we obtain a map from $B_{2}(R_{m}) \otimes R_{m} ^{\times}$ to $\Omega^{1} _{R}.  $ 

Since we know that $M_{m,r}$ is of $\star$-weight $r,$ in order to prove that  $L_{m,r}$ is of $\star$-weight $r,$ it suffices to prove the same for $L_{m,r}-M_{m,r}.$ This difference is equal to the sum of 
\begin{eqnarray}\label{stareq1}
- \sum _{m\leq j <r } \big(\frac{r-j}{j}d \ell i_{m,j} \otimes \ell _{r-j}- \ell i _{m,j} \otimes d \ell _{r-j} \big)
\end{eqnarray}
and 
\begin{eqnarray}\label{stareq2}
\sum _{m\leq j <r }\frac{r-j}{j}\big(\sum _{1\leq a \leq j-m}\big ((j-a) d \ell _a \wedge \ell _{j-a}+a d \ell _{j-a} \wedge \ell _{a} \big) \big)\circ \delta\otimes \ell _{r-j}.
\end{eqnarray}

Let us first look at the term $(j-a) d \ell _a \wedge \ell _{j-a}+a d \ell _{j-a} \wedge \ell _{a}.$ For any $u \wedge v \in \Lambda ^{2} tR_{\infty}$ and $\lambda \in R^{\times},$ $((j-a) d \ell _a \wedge \ell _{j-a}+a d \ell _{j-a} \wedge \ell _{a})(\lambda \star (u\wedge v))=$
\begin{eqnarray*}
 & &\big( (j-a) d (\lambda ^a\ell _a) \wedge \lambda ^{j-a}\ell _{j-a}+a d (\lambda ^{j-a}\ell _{j-a}) \wedge \lambda ^a\ell _{a}\big)((u\wedge v))\\
 &=&  \big( \lambda ^{j}((j-a) d \ell _a \wedge \ell _{j-a}+a d \ell _{j-a} \wedge \ell _{a})+ (j-a)a  (\ell _a \wedge \ell _{j-a}+\ell _{j-a } \wedge \ell _a ) \lambda ^{j-1} d\lambda\big)(u\wedge v)\\
 &=& \lambda ^{j}((j-a) d \ell _a \wedge \ell _{j-a}+a d \ell _{j-a} \wedge \ell _{a})(u \wedge v).
\end{eqnarray*}
 Therefore the term (\ref{stareq2}) is of $\star$-weight $r.$ 
 
Similarly,   
 $(r-j)d \ell i_{m,j} \otimes \ell _{r-j}- j\ell i _{m,j} \otimes d \ell _{r-j} 
 $ evaluated on $\lambda \star (u \otimes v)$ is equal to 
\begin{eqnarray*}
& &(r-j)d (\lambda ^j\ell i_{m,j}) \otimes \lambda ^{r-j}\ell _{r-j}- j\lambda ^j\ell i _{m,j} \otimes d (\lambda ^{r-j}\ell _{r-j}) \\
&=& \lambda ^r \big((r-j)d \ell i_{m,j} \otimes \ell _{r-j}- j\ell i _{m,j} \otimes d \ell _{r-j} \big)+(r-j)j(\ell i _{m,j} \otimes \ell _{r-j}-\ell i _{m,j} \otimes \ell _{r-j}) \lambda ^{r-1}d \lambda \\
&=&\lambda ^r \big((r-j)d \ell i_{m,j} \otimes \ell _{r-j}- j\ell i _{m,j} \otimes d \ell _{r-j} \big)
\end{eqnarray*} 
evaluated on $u \otimes v.$ This implies that the term (\ref{stareq1}) is of $\star$-weight $r$ and finishes the proof of the lemma. 
\end{proof}

 \begin{proposition}
 The map $L_{m,r}: B_{2} (R_{m}) \otimes R_{m} ^{\times} \to \Omega^{1} _{R}$ vanishes on the boundaries of the elements in $\mathbb{Q}[R_{m}^{\flat}]$ and hence induces a map 
 $$
  (B_{2} (R_{m}) \otimes R_{m} ^{\times})/im(\delta) \to \Omega^{1} _{R}, $$
 which by restriction gives the regulator map ${\rm H}^{2} (R_{m},\mathbb{Q}(3)) \to\Omega^{1} _{R}$ of $\star$-weight $r$ we were looking for. We continue to denote these two induced maps by the same notation $L_{m,r}.$
 \end{proposition}
 
 \begin{proof}
 We know that $M_{m,r}$ vanishes on the boundary $\delta(se^u)$  of elements $se^u \in \mathbb{Q}[R_\infty],$ with $u=u|_{m},$ by Proposition  \ref{wrongmap}. We also know  by the previous lemma that $L_{m,r}$ descends to a map on $B_{2}(R_{m} ) \otimes R_{m} ^{\times}.$ Therefore, in order to prove the statement, we only need to prove that 
 $$
L_{m,r}(\delta(se^u))= M_{m,r}(\delta(se^u)),
 $$ 
 for $u=u|_{m}. $ We first rewrite $L_{m,r}$ as the composition of $\delta \otimes id$ with 
\begin{eqnarray*}
& &\sum _{a+b=r \atop {m \leq a, \, 1 \leq  b}} b\cdot  \ell _a \wedge \ell_b \otimes d\ell _0-\sum _{a+b+c=r \atop {m \leq a+b, \atop {1 \leq a, \, b, \, c<m}}}\frac{c}{a+b} b \cdot d\ell _a \wedge \ell _b\otimes \ell _c\\
& &-\sum _{a+b+c=r \atop {m \leq a, \, 1 \leq b, \, c}}\frac{c}{a+b} b \cdot d(\ell _a \wedge \ell _b)\otimes \ell _c+\sum _{a+b+c=r \atop {m \leq a, \, 1 \leq b, \, c}}b \cdot \ell _a \wedge \ell _b \otimes d \ell _c.
\end{eqnarray*}
On the other hand, recall that  $M_{m,r}$ is the composition of $\delta \otimes id$ with 
\begin{eqnarray*}
& &\sum _{a+b=r \atop {m \leq a, \, 1 \leq  b}} b\cdot  \ell _a \wedge \ell_b \otimes d\ell _0-\sum _{a+b+c=r \atop {m \leq a+b, \atop {1 \leq a, \, b, \, c}}}\frac{c}{a+b} b \cdot d\ell _a \wedge \ell _b\otimes \ell _c.
\end{eqnarray*}
If we compare the two expressions we see that all of the terms match above except possibly the ones that correspond to the triples $(a,b,c)$ with $1 \leq a,\, b, \, c,$ $a+b+c=r,$ and $m \leq a$ or $m \leq b.$ By anti-symmetry, we may assume without loss of generality that $m \leq a.$ We need to compare the coefficients of the terms $d \ell _a \wedge \ell _b \otimes \ell _c,$ $ \ell _a \wedge d \ell _b \otimes \ell _c,$ and $ \ell _a \wedge  \ell _b \otimes d \ell _c,$ subject to the above constraints, in $L_{m,r}$ and $M_{m,r}.$  
 
The coefficient of $d \ell _a \wedge \ell _b \otimes \ell _c$ in $L_{m,r}$ and $M_{m,r}$  are both equal to $-\frac{cb}{a+b}.$ The coefficient of $ \ell _a \wedge d\ell _b \otimes \ell _c$ in $L_{m,r}$
is $\frac{-cb}{a+b}$ and in $M_{m,r},$ it is $\frac{ca}{a+b}.$ Finally, the coefficient of $ \ell _a \wedge  \ell _b \otimes d \ell _c$ in $L_{m,r}$ is $b,$ whereas in $M_{m,r}$ it is 0.  

We finally note that the values of  $\ell _a \wedge d\ell _b \otimes \ell _c$ and $ \ell _a \wedge \ell _c \otimes d\ell _b$ on $\delta(se^u)\otimes se^u$ are the same when $u=u|_m.$ Then the equality $\frac{-cb}{a+b}+c=\frac{ca}{a+b}$ finishes the proof.
\end{proof}
The following is a direct consequence of Lemma  \ref{lemma-descent-to-m} which we record for later reference. 
\begin{corollary}\label{comultiplication}
We have a commutative diagram 
$$
\xymatrix{
B_{2}(R_r) \otimes R_r ^{\times} \ar^{\delta\otimes id}[r] \ar@{->>}[d] & \Lambda ^2 R_r ^{\times} \otimes R_{r} ^{\times} \ar[d]\\
B_{2}(R_m) \otimes R_m ^{\times}  \ar^{\;\;\;\;\;\;\;L_{m,r}}[r]& \Omega^1 _{R}.
 }
$$
\end{corollary}

We expect  that the above maps combine to give an isomorphism between the infinitesimal part of the  cohomology of $R_{m}$ and the direct sum of the module of K\"{a}hler differentials, justifying the name of the regulator, c.f. \cite[Conjecture 1.15]{config}. However, at this point, we can only prove the surjectivity:  

\begin{proposition}\label{middle coh}
Suppose that $R$ is a regular local $\mathbb{Q}$-algebra and $2 \leq m$ as above. The direct sum of the $L_{m,r}$ induce a surjection: 
$$
\oplus_{m< r< 2m}L_{m,r}: {\rm H}^{2}(R_{m},\mathbb{Q}(3))^{\circ} \twoheadrightarrow \oplus _{m< r< 2m} \Omega^1 _{R}. 
$$
\end{proposition}

\begin{proof}
Suppose that  $\alpha \in B_{2}(R_{m})^{\circ} $ is in the part of kernel of the $\delta^{\circ}$ which is of $\star$-weight $r.$ By Theorem \ref{theorem-b2}, this part is isomorphic to $R$ via the restriction of the  map $\ell i_{m,r}.$ Computing the value of $L_{m,r}$ on $\alpha \otimes b,$ for $b \in R^{\times},$ we see that $L_{m,r}(\alpha\otimes b)=\ell i_{m,r}(\alpha)\frac{db}{b}.$
Since $\alpha \otimes b $ is in the kernel of $\delta,$ we see that the image of $L_{m,r}$ above is the additive group generated by the set $Rd\log (R^{\times}).$ Since $R$ is local this is equal to  $\Omega^{1} _{R}.$ This implies  the surjectivity. 
\end{proof}

\section{Construction of the maps from $\Lambda ^3 (R_{2m-1},(t^m))^{\times}$ to  
$\Omega^{1} _{R/k}$ }

For a ring $A$ and ideal $I,$ let $(A,I)^{\times}:=\{(a,b)| a, \, b \in A ^{\times},\; a-b \in I \},$ and let $\pi_i:(A,I)^{\times} \to A^{\times},$ for $i=1,2$  denote the two projections.  In this section we will define a map 
$\omega_{m,r}:\Lambda ^{3}(R_{r},(t^m))^{\times}\to \Omega^{1} _{R/k},$ which will fundamentally depend on the assumption that $R/k$ is a 
smooth,  $k$-algebra of relative dimension 1.

\subsection{Description of the construction} Assume that $R/k$ is a smooth,  $k$-algebra of relative dimension 1. Using the map $L_{m,r},$ we will construct a map $\omega_{m,r}:\Lambda ^{3}(R_{r},(t^m))^{\times}\to \Omega^{1} _{R/k}.$   Let us denote the composition of $L_{m,r}$ with the canonical projection $\Omega^1 _{R} \to \Omega^1 _{R/k}$ by $\underline{L}_{m,r}.$  Let  $s:\Lambda ^{3}(R_{r},(t^m))^{\times} \to \Lambda ^{3}R_{r} ^{\times},$ be the map given by $s:=\Lambda ^3\pi_{1} -\Lambda ^{3}\pi_2.$    This map factors through  $I_{m,r}:=im \big((1+(t^m)) \otimes \Lambda ^{2}R_{r} ^{\times} \big) \subseteq (\Lambda ^{3}R_{r} ^{\times})^{\circ}.$ We will define a map 
$$
\Omega_{m,r}: I_{m,r} \to \Omega^1_{R/k},
$$
and then let 
$$
\omega_{m,r}:=\Omega_{m,r}\circ s.
$$

Given $\alpha \in I_{m,r},$ there exists $\varepsilon \in im((1+(t^m))\otimes k^{\times} \otimes R_{r}^{\times}) \subseteq (\Lambda ^{3} R_{r} ^{\times})^{\circ},$ such that 
$\alpha -\varepsilon =\delta_r (\gamma) $ for some $\gamma \in (B_{2}(R_{r})\otimes R_{r} ^{\times})^{\circ}.$ We then define $\Omega_{m,r}(\alpha):=\underline{L}_{m,r}(\gamma|_{t^m}) \in \Omega^{1} _{R/k}.$ The map is well-defined only after we pass to the quotient consisting of relative differentials. 

\subsection{Proof that $\Omega_{m,r}$ is well-defined}

In this section, we will justify the construction we made in the above paragraph in several steps. We assume that $R$ is a smooth local $k$-algebra of relative dimension 1. 

\begin{lemma}\label{redtok} 
  For any  $\alpha \in  (\Lambda^{3}R_{r}^{\times})^{\circ},$ there exists $\varepsilon \in   im(\Lambda^{2}R_{r}^{\times}\otimes k^{\times})\subseteq (\Lambda^{3}R_{r}^{\times})^{\circ}$ such that $\alpha -\varepsilon$ lies in the image of   $(B_{2}(R_{r}) \otimes R_{r} ^{\times})^{\circ}$ in  $ (\Lambda^{3}R_{r}^{\times})^{\circ}.$ Moreover, if 
  $$\alpha \in  im((1+t^mR_{r})^{\times } \otimes \Lambda^{2}R_{r}^{\times}) \subseteq (\Lambda^{3}R_{r}^{\times})^{\circ}$$ 
  then we can choose 
  $$\varepsilon \in   im( (1+t^mR_{r})^{\times } \otimes R_{r}^{\times}\otimes k^{\times})\subseteq (\Lambda^{3}R_{r}^{\times})^{\circ}$$ 
  such that $\alpha -\varepsilon$ lies in the image of   $(B_{2}(R_{r}) \otimes R_{r} ^{\times})^{\circ}$ in  $ (\Lambda^{3}R_{r}^{\times})^{\circ}.$
\end{lemma}

\begin{proof}

The infinitesimal part of the cokernel of the map 
$$
 B_{2}(R_{r})\otimes R_{r}^{\times} \to \Lambda^{3}R_{r}^{\times}
$$
is  $K_{3}^{M}(R_{r})^{\circ}$ which is isomorphic to   $\oplus _{1 \leq i <r}t^i\Omega^{2} _{R} $ via the map from $ \Lambda^{3}R_{r}^{\times},$ whose $i$-th coordinate  is given by 
\begin{eqnarray}\label{cokernel}
res_{t=0} \frac{1}{t^i} d \log(y_1)\wedge d \log(y_2)\wedge d \log(y_3),
\end{eqnarray}
by Corollary \ref{cor milnor}. 
   Further, by the assumption on smoothness of dimension 1, we conclude that the natural map 
$$
\Omega^{1} _{R }\otimes _k\Omega^{1} _{k} \to \Omega^{2} _{R }
$$
is surjective. Since we assume  that $R$ is local, the map
$$
R\otimes _{\mathbb{Z}}R^{\times} \to \Omega^{1} _{R}, 
$$
which sends $a \otimes b$ to $a\cdot d \log b,$ is surjective. 

Note that the image of $e^{t^iu} \wedge v \wedge \lambda,$ with $u  \in R,$ $v \in R^{\times}$ and $\lambda \in k^\times$ under the $i$-th map in (\ref{cokernel})  is  $ i \cdot u \cdot d\log (v) \wedge d \log (\lambda)$ and under the coordinate $j$ maps in  (\ref{cokernel}) with $j\neq i,$ the image is 0.
Together with the above, this shows that the map 
$$
(\Lambda^{2} R_{r} ^{\times} \otimes k^{\times})^{\circ} \to (\Lambda^{3} R_{r} ^{\times})^{\circ} \to \oplus _{1 \leq i <r} t^i \Omega^{2} _{R} 
$$
is surjective and hence proves the first statement. 

 For the second statement,  note that if $\alpha \in  im((1+t^m R_{r})^{\times } \otimes \Lambda^{2}R_{r}^{\times}) \subseteq (\Lambda ^{3} R_{r} ^{\times})^{\circ}$ then its image in $\oplus _{1 \leq i <r}t^i\Omega^{2} _{R }$ lands in the summand $\oplus _{m \leq i <r}t^i\Omega^{2} _{R }$. Since by the above discussion, we also see that the composition 
 $$
(1+t^mR_{r})^{\times} \otimes R_{r} ^{\times} \otimes k^{\times} \to (\Lambda^{3} R_{r} ^{\times})^{\circ} \to \oplus _{m \leq i <r}t^i\Omega^{2} _{R }
 $$
 is surjective, the second statement similarly follows. 
 \end{proof}

  Note that the map $L_{m,r}: B_{2} (R_m) \otimes R_{m} ^{\times} \to \Omega^{1} _{R},$ and the map $\underline{L}_{m,r}:B_{2} (R_m) \otimes R_{m} ^{\times} \to \Omega^{1} _{R/k}$ obtained by composing $L_{m,r}$ with the projection from $\Omega^{1} _{R} $ to $\Omega^{1} _{R/k} $ are of $\star$-weight $r.$ Since the $\star$-weights of  ${\rm H}^2(R_{r},\mathbb{Q}(3))$ are expected to be between $r+1$ and $2r-1$ as we mentioned in the paragraph preceding Proposition \ref{middle coh},  one would expect the following lemma: 
  \begin{lemma}\label{vanish_on_kernel}
The composition
\[
 \begin{CD}
ker(\delta _r) \subseteq B_{2}(R_{r}) \otimes R_{r} ^{\times} @>{|_{t^m}}>> B_{2}(R_m) \otimes R_m ^{\times}  @>{L_{m,r}}>> \Omega^{1} _{R} 
\end{CD}
\]
is 0. 
\end{lemma}

\begin{proof} This follows immediately from Corollary \ref{comultiplication}.
\end{proof}

\begin{remark} We emphasize that in the above lemma our assumption is that $\delta_r(\gamma)=0$ in $\Lambda ^{3} R_{r} ^{\times}$ and {\it not the weaker assumption} that $\delta_m(\gamma|_{t^m})=0.$ In fact, there are elements such that $\delta_m(\gamma|_{t^m})=0$ and   $L_{m,r}(\gamma|_{t^m}) \neq 0.$ 
\end{remark}

\begin{lemma}
Suppose that $\gamma \in B_{2}(R_r) \otimes R_r ^{\times}$   such that 
 $$ \delta_r(\gamma) \in  im(\big(\oplus_{0 \leq s<r}   (1+t^sR_r)^{\times } \otimes (1+t^{r-s}R_r)^{\times}    \big) \otimes k^{\times})\subseteq \Lambda ^3 R_r ^{\times}$$
 then $\underline{L}_{m,r}(\gamma|_{t^m})=0.$ 
\end{lemma}

\begin{proof} 
First let $\alpha:=e^{ut^i}\wedge e^{vt^j}\wedge \lambda \in \Lambda ^{3} R_{r} ^{\times}$ with  $u, \, v \in R$ and $\lambda \in k^{\times}$ and  $r \leq i+j.$ We have  $e^{ut^i}\wedge e^{vt^j} \in (\Lambda ^{2} R_{r} ^{\times})^{\circ}$ and 
$$
res_{t} \frac{1}{t^a}d \log (e^{ut^i})\wedge d \log (e^{vt^j})=res_{t} \frac{1}{t^a}d  (ut^i)\wedge d (vt^j)=0 \in \Omega^{1} _{R},
$$
for all $1\leq a  < r,$ since $r \leq i+j .$ This implies that there is $\alpha_0 \in B_{2}(R_{r})$ such that $\delta_r (\alpha_0)= e^{ut^i}\wedge e^{vt^j},$ and hence $\delta_r(\alpha_0 \otimes \lambda)=\alpha.$ 

Let us compute $L_{m,r}((\alpha_0 \otimes \lambda)|_{t^m}).$ Since $\ell _i (\lambda) =0, $ for $0<i$ by the formula for $L_{m,r}$ we see that  
$$
L_{m,r}((\alpha_0 \otimes \lambda)|_{t^m})=\ell i _{m,r}(\alpha_0 |_{t^m}) \cdot d \log (\lambda) \in \Omega^{1} _{R}.
 $$
 This expression vanishes in $\Omega^{1} _{R/k}$ and therefore $\underline{L}_{m,r}((\alpha_{0} \otimes \lambda)|_{t^m})=0.$ Taking the sum of expressions such as above, we deduce that if 
 $$
 \alpha \in  im(\big(\oplus_{0 \leq s<r}   (1+t^sR_r)^{\times } \otimes (1+t^{r-s}R_r)^{\times}    \big) \otimes k^{\times})\subseteq \Lambda ^3 R_r ^{\times}
 $$ then there is a $\tilde{\alpha} \in B_{2}(R_{r}) \otimes R_{r} ^{\times} $ such that $\delta_r (\tilde{\alpha})=\alpha$  and $\underline{L}_{m,r}(\tilde{\alpha}|_{t^m})=0.$ 
 
 Applying this to $\alpha:=\delta_r(\gamma) $ we deduce that there exists $\tilde{\alpha} \in B_{2}(R_{r}) \otimes R_{r} ^{\times} $ such that $\delta_r (\tilde{\alpha})=\delta_r(\gamma)$  and $\underline{L}_{m,r}(\tilde{\alpha}|_{t^m})=0.$
Then we have $\underline{L}_{m,r}(\gamma|_{t^m})=\underline{L}_{m,r}(\tilde{\alpha}|_{t^m})+\underline{L}_{m,r}((\gamma-\tilde{\alpha})|_{t^m})=\underline{L}_{m,r}((\gamma-\tilde{\alpha})|_{t^m}).$ Since $\delta_r(\gamma-\tilde{\alpha})=0,$ the last expression is 0 by Lemma \ref{vanish_on_kernel}. 
\end{proof}

\begin{lemma}\label{strongvanlemma}
Suppose that $\gamma \in B_{2}(R_r) \otimes R_r ^{\times}$   such that $\delta_r(\gamma) \in  im( (\Lambda ^2 R_r^{\times})^{\circ}\otimes k^{\times}) \subseteq \Lambda^{3} R_{r} ^{\times} $ then $\underline{L}_{m,r}(\gamma|_{t^m})=0.$ 
\end{lemma}

\begin{proof} 

Suppose that $\gamma$ is as in the statement of the lemma. Fix some $a \in \mathbb{Z}_{>1}.$  We inductively define $\gamma ^{[i]}$ as follows. Let $\gamma^{[-1]}=\gamma,$ and 
$$
\gamma ^{[i]}:=a \star \gamma ^{[i-1]}-a^{i} \gamma^{[i-1]},
$$
for $0\leq i <r.$ Since $L_{m,r}$ is of weight $r$, $L_{m,r}(\gamma ^{[i]}|_{t^m})=(a^r-a^i)L_{m,r}(\gamma ^{[i-1]}|_{t^m}).$ Therefore proving that $\underline{L}_{m,r}(\gamma|_{t^m})=0$ is equivalent to proving that $\underline{L}_{m,r}(\gamma ^{[r-1]}|_{t^m})=0.$ On the other hand, 
$$
\delta (\gamma ^{[r-1]}) \subseteq im(\big(\oplus_{0 \leq s<r}   (1+t^sR_r)^{\times } \otimes (1+t^{r-s}R_r)^{\times}    \big) \otimes k^{\times}),$$
and therefore the previous lemma implies that $\underline{L}_{m,r}(\gamma ^{[r-1]}|_{t^m})=0.$
\end{proof}
These lemmas together finish the proof that $\Omega_{m,r}$ is well-defined as follows. Starting with $\alpha \in I_{m,r},$ we know, by Lemma \ref{redtok}, that there exists $\varepsilon \in im((1+(t^m))\otimes k^{\times} \otimes R_{r}^{\times}) \subseteq (\Lambda ^{3} R_{r} ^{\times})^{\circ},$ such that 
$\alpha -\varepsilon =\delta_r (\gamma) $ for some $\gamma \in (B_{2}(R_{r})\otimes R_{r} ^{\times})^{\circ}.$ In order to see that $\Omega_{m,r}(\alpha):=\underline{L}_{m,r}(\gamma|_{t^m}) \in \Omega^{1} _{R/k}$ is well-defined, suppose that $\varepsilon'  \in im((1+(t^m))\otimes k^{\times} \otimes R_{r}^{\times}) \subseteq (\Lambda ^{3} R_{r} ^{\times})^{\circ}$ and $\gamma' \in (B_{2}(R_{r})\otimes R_{r} ^{\times})^{\circ}$ are another such choices. Then $\delta_r(\gamma'-\gamma)=\varepsilon -\varepsilon' \in im( (\Lambda ^2 R_r^{\times})^{\circ}\otimes k^{\times}) \subseteq \Lambda^{3} R_{r} ^{\times}   $ and hence $\underline{L}_{m,r}((\gamma'-\gamma)|_{t^m})=0$ by Lemma \ref{strongvanlemma}.
 
 The following proposition gives an explicit expression for $\Omega_{m,r}.$ 
 
 \begin{proposition}\label{formula Omega}
Suppose that $x \geq m$ and $x+y+z=r,$ with $y$ or $z$ possibly 0, then 
$$
\Omega_{m,r}(e^{at^x}\wedge e^{bt^y}\wedge e^{ct^z})=a(yb\cdot dc- zc \cdot db).
$$ Here when $y=0$,   we let $e^b$ denote an arbitrary element of $R^{\times }$ and  $yb$ denote 0 and $db$ denote $d \log (e^b)=\frac{d(e^b)}{e^b}.$  Note that when $y=0,$ the expression $b$ does not make sense. Similarly, for $zc$ and $dc$ when $z=0.$   
 \end{proposition}

\begin{proof}

(i) Case when $y=z=0.$ In this case, we need to compute the image of $e^{at^r}\wedge \beta \wedge \gamma$ under $\Omega_{m,r},$ where $\beta, \, \gamma \in R^{\times}.$ The image of $e^{at^r}\wedge \beta$ in $\oplus _{1 \leq i \leq r-1} t^i\Omega^1 _{R}$ is equal to 0. Therefore, there is $\alpha \in B_{2} ^{\circ}(R_r)$ such that $\delta_{r}(\alpha)=e^{at^r}\wedge \beta.$ Then, by definition,        
\begin{eqnarray}\label{local omega form}
\Omega_{m,r}(e^{at^r}\wedge \beta \wedge \gamma)=\underline{L}_{m,r}((\alpha \otimes \gamma)|_{t^m}).
\end{eqnarray}
On the other hand, by (\ref{eqn lmr}), $L_{m,r}(\alpha|_{t^m} \otimes \gamma)=\ell i _{m,r}(\alpha|_{t^m}) \frac{d\gamma}{\gamma}.$  By the expression (\ref{twodilogs}) for $\ell i_{m,r},$ we have 
$$
\ell i _{m,r}(\alpha|_{t_m})=\Big(\Lambda ^{2} \log^{\circ} (\delta_r(\alpha))| \sum _{1\leq i \leq r-m} it^{r-i}\wedge t^i\Big)
=0,
$$
since $\Lambda ^{2} \log^{\circ} (\delta_r(\alpha))=\Lambda ^{2} \log^{\circ}(e^{at^r}\wedge \beta)=0.$ By the above formula (\ref{local omega form}), this implies that  $\Omega_{m,r}(e^{at^r}\wedge \beta \wedge \gamma)=0$ as we wanted to show. 

(ii) Case when  $y \neq 0$ and $z=0.$ In this case we try to compute the image of $e^{at^x}\wedge e^{bt^y}\wedge \gamma$ under $\Omega_{m,r}.$ Here we assume that $\gamma \in R^{\times}$ and $x+y=r,$ with $x\geq m.$ By exactly the same argument as above, we deduce that there exists $\alpha \in B_{2} ^{\circ}(R_{r})$  such that $\delta _r(\alpha)=e^{at^x}\wedge e^{bt^y}$ and we have 
\begin{eqnarray*}
\Omega_{m,r}(e^{at^x}\wedge e^{bt^y} \wedge \gamma)=\underline{L}_{m,r}((\alpha \otimes \gamma)|_{t^m})=\ell i _{m,r}(\alpha|_{t^m}) \frac{d\gamma}{\gamma}.
\end{eqnarray*}
Since $\delta _r(\alpha)=e^{at^x}\wedge e^{bt^y}$
$$
\ell i _{m,r}(\alpha|_{t_m})=\Big(\Lambda ^{2} \log^{\circ} (\delta_r(\alpha))| \sum _{1\leq i \leq r-m} it^{r-i}\wedge t^i\Big)
=yab.
$$
This exactly coincides with the expression in the statement of the proposition. 

(iii) Case when $y\neq 0$ and $z \neq 0.$ Note that, by localizing, we may assume that $R$ is local. Moreover, since both sides of the expression are linear in $a, \, b$ and $c,$ we may assume without loss of generality that $a, \,b, \, c \in R^{\times}.$  Since any element in a local ring can be written as a sum of  units. 

If $\theta:=e^{at^x}\wedge e^{
bt^y}$ then its image in $\oplus _{1 \leq i \leq r-1} t^i\Omega^1 _{R}$ is equal to 
$$
 \oplus _{1 \leq i \leq r-1} res_{t}\frac{1}{t^i}(\Lambda ^ {2} d \log (e^{at^x}\wedge e^{bt^y})),
$$
which only has a non-zero component in degree $x+y$ equal to $yb\cdot da- xa \cdot db.$ 

If we compute the image of $\varphi:= \frac{x}{x+y}e^{abt^{x+y}}\wedge b -\frac{y}{x+y}e^{ab t^{x+y}}\wedge a$ in the same group, we obtain the same element. Therefore  $\theta -\varphi$  lies in the image of $B_{2}(R_{r}).$ Suppose that $\gamma_{0} \in B_{2}(R_{r})$ such that $\delta (\gamma_{0})=\theta -\varphi.$  Since $e^{abt^{x+y}}\wedge e^{ct^z}$ has weight $r$, there is $\varepsilon_{0} \in B_{2} (R_{r}) $ such that $\delta(\varepsilon_{0})=e^{abt^{x+y}}\wedge e^{ct^z}.$ 

We now write 
$$
e^{at^x}\wedge e^{bt^y}\wedge e^{ct^z}=(\theta -\varphi)\wedge e^{ct^z}+ \varphi \wedge e^{ct^z}=\delta (\gamma_{0} \otimes e^{ct^z})-\frac{x}{x+y}\delta(\varepsilon_0 \otimes b) +\frac{y}{x+y} \delta(\varepsilon_0 \otimes a).
$$
 By the definition of $\Omega_{m,r},$ we have $\Omega_{m,r}(e^{at^x}\wedge e^{bt^y}\wedge e^{ct^z})=$
 $$
\underline{L}_{m,r} (\gamma_{0}|_{t^m} \otimes e^{ct^z})-\frac{x}{x+y}\underline{L}_{m,r}(\varepsilon_0|_{t^m} \otimes b) +\frac{y}{x+y} \underline{L}_{m,r}(\varepsilon_0|_{t^m} \otimes a).
 $$ 
 By definition, $\underline{L}_{m,r}(\varepsilon_0|_{t^m} \otimes b)=$
 $$\ell i _{m,r}(\varepsilon_0|_{t^m} ) d \log (b)=(\Lambda ^{2} \log ^{\circ}\delta (\varepsilon_{0})|\sum _{m \leq i<r} (r-i)t^i\wedge t^{r-i} )d \log (b)=zabc \cdot d\log (b)=azc\cdot db.
 $$
By the same argument,
$\underline{L}_{m,r}(\varepsilon_0|_{t^m} \otimes a)=bzc \cdot da.$ 
 
In order to compute $\underline{L}_{m,r} (\gamma_{0}|_{t^m} \otimes e^{ct^z}),$ first note that, by the definition of $L_{m,r},$ we have 
$$
\underline{L}_{m,r} (\gamma_{0}|_{t^m} \otimes e^{ct^z})=-\frac{z}{x+y}d\ell i_{m,x+y} (\gamma_0|_{t^m})\cdot c+\ell i_{m,x+y} (\gamma_0|_{t^m}) \cdot d c.
$$
Since $\ell i_{m,x+y} (\gamma_0|_{t^m})=(\Lambda ^{2} \log ^{\circ}\delta (\gamma_{0})|\sum _{m \leq i<x+y} (x+y-i)t^i\wedge t^{x+y-i} )=yab,$ we have 
$$
\underline{L}_{m,r} (\gamma_{0}|_{t^m} \otimes e^{ct^z})=-\frac{zy}{x+y}c \cdot d(ab)+yab \cdot dc.
$$
Combining all of these gives, $\Omega_{m,r}(e^{at^x}\wedge e^{bt^y}\wedge e^{ct^z})=$
$$
-\frac{zy}{x+y}c \cdot d(ab)+yab \cdot dc-\frac{xzac}{x+y} \cdot db +\frac{yzbc}{x+y}da=a(yb\cdot dc-zc\cdot db).
$$
This finishes the proof of the proposition. 
\end{proof}

\begin{definition}\label{defn om}
We define $\omega_{m,r}: \Lambda ^{3}(R_{r},(t^m))^{\times}\to \Omega^{1} _{R/k}$ as the composition $\Omega_{m,r} \circ s$ of $s:\Lambda ^{3}(R_{r},(t^m))^{\times} \to I_{m,r},$ and $\Omega_{m,r}: I_{m,r} \to \Omega^1 _{R/k}.$
\end{definition}

\subsection{Behaviour of $\omega_{m,r}$ with respect to automorphisms of $R_{2m-1}$ which are identity modulo $(t^m)$}

In this section, we will  show the  invariance of $\Omega_{m,r}$ with respect to reparametrizations of $R_r$ that are identity on the reduction to $R_m.$ In order to do this, we will need to make an explicit computation on $k'((s))_{\infty},$ where $k'$ is a finite extension on $k.$  In order to make the formulas concise and intuitive, we will use several notational conventions as follows. If $a \in k'((s)),$ we let $a'=a^{(1)} 
\in k'((s))$ denote its derivative with respect to $s$ and $a^{(n+1)}=(a^{(n)})'.$ Similarly, we let $e^a$ denote an arbitrary {\it non-zero} element in $k'((s)) $ and $a'=a^{(1)}:=\frac{(e^a)'}{e^a}$ and  $a^{(n+1)}=(a^{(n)})'.$ This  notation is   intuitive in the sense that, if one thinks of $a$ as $\log (e^a)$ then $a'$ is the logarithmic derivative of $e^a.$ With these conventions, we will state the following basic lemma. 

\begin{lemma}
Let $\sigma$ be the automorphism of the $k_{\infty}$ algebra $k'((s))_{\infty},$ which is the identity map modulo $(t)$ and has the property that  $\sigma(s)=s+\alpha t^w,$ with $w \geq 1$ and $\alpha\in k((s)),$ then $\sigma(e^{at^x})=e^{\sum_{0\leq i}\frac{\alpha^i
a^{(i)}}{i!}t^{x+iw}}.$
 \end{lemma}

\begin{proof}
First note that such an automorphism should be identity on $k'$ since $k'/k$ is \'{e}tale, the map is identity on $k$ and is also identity from $k'$ to $k'((s))_\infty/(t)=k'((s)).$  Therefore, $\sigma$ is an automorphism of the $k'_{\infty}$-algebra $k'((s))_{\infty}.$  
 The proof is then separated into two cases, when $x=0$ and when $x\neq 0.$ In both cases, the statement follows from the Taylor expansion formula. 
\end{proof}

\begin{lemma}\label{easymainlemma}
 Let $\sigma$ be the automorphism of the $k_{\infty}$ algebra $k'((s))_{\infty}$ given by $\sigma(s)=s+\alpha t^w,$  with $m \leq w$ and identity modulo $(t).$    Then we have,
$$
\Omega_{m,r} \big(\frac{\sigma(e^{at^i}\wedge e^{bt^j}\wedge e^{ct^k} )}{e^{at^i}\wedge e^{bt^j}\wedge e^{ct^k}}\big)=0,
$$
for $0<i+j+k.$ 
\end{lemma}

\begin{proof} Since $m\leq w,$ $0<i+j+k$ and $m<r<2m,$  
the weight $r$ terms of $\frac{\sigma(e^{at^i}\wedge e^{bt^j}\wedge e^{ct^k} )}{e^{at^i}\wedge e^{bt^j}\wedge e^{ct^k}}$ are possibly non-zero only when   $i+j+k+w = r$ and in this case they are given by 
$$
e^{\alpha a't^{i+w}}\wedge e^{bt^j}\wedge e^{ct^k}+e^{at^{i}}\wedge e^{ \alpha b't^{j+w}}\wedge e^{ct^k}+e^{at^{i}}\wedge e^{bt^j}\wedge e^{\alpha c't^{k+w}}.
$$
By the formula in Proposition \ref{formula Omega}, the above sum is sent to 
$$
\alpha (a' (jbc'-kc b')-b'(iac'-kc a')+c'(iab'-jba'))ds=0.
$$
\end{proof}

\begin{corollary}\label{cor semi path}
Let $\sigma$ be any automorphism of $R_r$ as a $k_r$-algebra, which reduces to identity on $R_m,$ then $\omega_{m,r}\circ \Lambda^3\sigma=\omega_{m,r}.$
\end{corollary}

\begin{proof}
This follows by the corresponding statement for $\Omega_{m,r}.$ This in turn reduces to Lemma \ref{easymainlemma} after localizing and completing. 
 \end{proof}

\begin{definition}\label{defn omega path}
If  $\pazocal{R}/k_{r}$ is a smooth  $k_{r}$-algebra of relative dimension 1. 
We defined the  map  
$$
\omega_{m,r}:\Lambda^3 (\underline{\pazocal{R}}_{r},(t^m) )^{\times} \to \Omega^{1} _{\underline{\pazocal{R}}/k},
$$
as the composition $\Omega_{m,r}\circ s,$ where $\underline{\pazocal{R}}$ is the reduction of $\pazocal{R}$ modulo $(t)$ and $\underline{\pazocal{R}}_{r}:=\underline{\pazocal{R}} \times _k k_{r}.$   
Let $\tau: \underline{\pazocal{R}}_{r} \to \pazocal{R}$ be a splitting, that is an isomorphism of $k_{r}$-algebras which is the  identity map modulo $(t).$ By transport of structure, this gives a map $$
\omega_{m,r,\tau}:\Lambda^3 (\pazocal{R},(t^m) )^{\times} \to \Omega^{1} _{\underline{\pazocal{R}}/k}.
$$
Suppose that $\tau '$ is  another such  splitting which agrees with $\tau$ modulo $(t^m).$  Applying Corollary \ref{cor semi path}  to ${\tau '} ^{-1}\circ\tau,$ we deduce that $\omega_{m,r,\tau}=\omega_{m,r,\tau'}.$
Therefore, if  $\sigma:\underline{\pazocal{R}}_{m} \to \pazocal{R}/(t^m) $ is a splitting of the reduction $\pazocal{R}/(t^m)$ of $\pazocal{R}$ then $\omega_{m,r,\sigma}$ is unambiguously defined as $\omega_{m,r,\tau},$ where $\tau$ is any splitting of $\pazocal{R}$ that reduces to $\sigma$ modulo $(t^m).$ 

\end{definition}

Recall the relative version of the Bloch group from \cite[\textsection 2.4.8]{unv3}. If $A$ is a ring with ideal $I, $ let $(A,I)^{\flat}:= \{ (\tilde{a},\hat{a}) \in (A,I)^{\times}| (1-\tilde{a},1-\hat{a}) \in (A,I)^{\times}\}.$ Then  the relative Bloch group $B_{2}(A,I)$ is defined as the abelian group  generated by the symbols $[ (\tilde{a},\hat{a})]$ for every $(\tilde{a},\hat{a}) \in (A,I)^{\flat},$ modulo the relations generated by the analog of the five term relation for the dilogarithm: 
$$
 [(\tilde{x},\hat{x})]-[(\tilde{y},\hat{y})]+[ (\tilde{y}/\tilde{x},\hat{y}/\hat{x})]-[(\frac{1-\tilde{x}^{-1}}{1-\tilde{y}^{-1}},\frac{1-\hat{x}^{-1}}{1-\hat{y}^{-1}})]+[(\frac{1-\tilde{x}}{1-\tilde{y}},\frac{1-\hat{x}}{1-\hat{y}})]
$$
for every $(\tilde{x},\hat{x}),(\tilde{y},\hat{y}) \in (A,I)^{\flat}$ such that $(\tilde{x}-\tilde{y},\hat{x}-\hat{y})\in (A,I)^{\times}.$ As in the classical case, we obtain a complex 
$
\delta: B_{2} (A,I) \to \Lambda^{2} (A,I)^{\times},
$
which sends $(\tilde{a},\hat{a})$ to $(1-\tilde{a},1-\hat{a})\wedge (\tilde{a},\hat{a}).$  As usual, abusing the notation,  we will denote the induced map $\delta\otimes id:
 B_{2} (A,I) \otimes (A,I)^{\times} \to \Lambda^{3} (A,I)^{\times},
$
also by $\delta.$ With these definitions,  we have the following expected property of the map $\omega_{m,r,\tau}.$ 
\begin{proposition}\label{value of omega on a boundary}
For a splitting $\sigma$ of $\pazocal{R}/(t^m),$ the above map $\omega_{m,r,\sigma}$ vanishes on the image of $B_{2}(\pazocal{R},(t^m)) \otimes (\pazocal{R},(t^m))^{\times}$ in $\Lambda^3 (\pazocal{R},(t^m) )^{\times}$ under $\delta.$ 
\end{proposition}

\begin{proof}
By the definition of $\omega_{m,r,\sigma},$ we easily reduce to the split case where $\pazocal{R}=R_{r}.$  We need to prove that $\Omega_{m,r}$ vanishes on the following two types of elements:
$$
\delta_r([\hat{f} ]\otimes \hat{g})-\delta_r([\tilde{f}]\otimes \hat{g})
\;\; \;{\rm and} \;\;\;
\delta_r([\hat{f}]\otimes \hat{g})-\delta_r([\hat{f}]\otimes \tilde{g}),
$$
where $\hat{f}, \, \tilde{f} \in \pazocal{R} ^{\flat}$ have the same reduction modulo $(t^m)$ and $\hat{g}, \, \tilde{g} \in \pazocal{R} ^{\times}$ have the same reduction modulo $(t^m).$ 
By the definition of $\Omega_{m,r},$ its value on the first and the second expressions are respectively: 
$$
\underline{L}_{m,r}(([\hat{f} ]\otimes \hat{g}-[\tilde{f} ]\otimes \hat{g})|_{t^m})=0,
$$
since $\hat{f}|_{t^m}=\tilde{f}|_{t^m}, $ and 
$$
\underline{L}_{m,r}(([\hat{f} ]\otimes \hat{g}-[\hat{f} ]\otimes \tilde{g})|_{t^m})=0,
$$
since $\hat{g}|_{t^m}=\tilde{g}|_{t^m}. $
\end{proof}

\subsection{Behaviour of $res(\omega_{m,r,\sigma})$ with respect to automorphisms of $R_{m}$} In order to proceed with our construction, we need an object such as the 1-form in \cite{unv3} which  controls the effect of changing splittings.   This object in $\star$-weight $r$ will be  constructed below by using  $\omega_{m,r}.$ On the other hand, this objects {\it does depend} on the choice of splittings if these splittings are different modulo $(t^{m}),$ when $r>m+1.$  In the modulus $m=2$ case the only possible  $r$ is 3 so this situation does not occur in \cite{unv3}. In the current case of higher modulus, we will see that  the  {\it residues} of the 1-form $\omega_{m,r} $  is invariant under the automorphisms of $R_{m}$ which are identity modulo $(t),$ which will imply that the residue can be defined independent of various choices.  We will see that this will be enough for constructing the Chow dilogarithm of higher modulus. We will again start with an explicit computation on $k'((s))_{\infty}.$ 

\begin{proposition}\label{main comp residue}
Suppose that $\sigma$ is the  automorphism of $k'((s))_{\infty}$ as a $k_{\infty}$-algebra such that  $\sigma(s)=s+\alpha t^w,$ with $w\geq 1$ and $\alpha \in k'((s)),$ and which is identity modulo $(t).$   Consider the element 
$
e^{at^x} \wedge e^{bt^y} \wedge e^{ct^z},
$ with $m \leq x.$ If $r-(x+y+z)>0,$ and is divisible by $w,$ let $q=\frac{r-(x+y+z)}{w}.$ Then 
$
\Omega_{m,r} \big(\frac{\sigma(e^{at^x}\wedge e^{bt^y}\wedge e^{ct^z} )}{e^{at^x}\wedge e^{bt^y}\wedge e^{ct^z}}\big)
$ is equal to 
\begin{eqnarray}\label{antider}
\;\;\;\;\;\;\;\;\;\;d\Big(\frac{\alpha^q}{q!}  \sum _{0 \leq k \leq q-1} a^{(k)} \binom{q-1}{k} \sum_{i+j=q-k}\big( \binom{q-k-1}{i} y b^{(i)}c^{(j)}- \binom{q-k-1}{j} z b^{(i)}c^{(j)} \big)  \Big)
 \end{eqnarray}
 Otherwise, $
\Omega_{m,r} \big(\frac{\sigma(e^{at^x}\wedge e^{bt^y}\wedge e^{ct^z} )}{e^{at^x}\wedge e^{bt^y}\wedge e^{ct^z}}\big)=0.
$ 
\end{proposition}

\begin{proof}
First note that by the above lemma   
$$R:=\frac{\sigma(e^{at^x}\wedge e^{bt^y}\wedge e^{ct^z} )}{e^{at^x}\wedge e^{bt^y}\wedge e^{ct^z}}=\frac{e^{\sum_{0\leq i}\frac{\alpha^i
a^{(i)}}{i!}t^{x+iw}}\wedge e^{\sum_{0\leq i}\frac{\alpha^i
b^{(i)}}{i!}t^{y+iw}} \wedge e^{\sum_{0\leq i}\frac{\alpha^i
c^{(i)}}{i!}t^{z+iw}}}{e^{at^x}\wedge e^{bt^y}\wedge e^{ct^z}}$$ and hence  if $r-(x+y+z)\leq 0$ or $w \nmid r-(x+y+z)$ then $R$ does not have a component of weight $r$ and  $\Omega_{m,r}(R)=0.$ 

Suppose then that $r-(x+y+z)>0,$  $w| (r-(x+y+z))$ and let $q:=\frac{r-(x+y+z)}{w}$ as in the statement of the proposition. In this case the weight $r$ term of $R$ is given as 
$$
\sum_{i+j+k=q}e^{\frac{\alpha^ka^{(k)}}{k!}t^{x+kw}}\wedge e^{\frac{\alpha^ib^{(i)}}{i!}t^{y+iw}}\wedge e^{\frac{\alpha^j c^{(j)}}{j!}t^{z+jw}}.
$$
This implies that $\Omega_{m,r}(R)$ is equal to 
$$
\sum_{i+j+k=q}\frac{\alpha^ka^{(k)}}{k!} \Big( (y+iw)\frac{\alpha^ib^{(i)}}{i!}\big(\frac{\alpha^j c^{(j)}}{j!}\big)'-(z+jw) \big(\frac{\alpha^ib^{(i)}}{i!} \big) ' \frac{\alpha^j c^{(j)}}{j!} \Big)ds.
$$
We first claim that the expression above does not depend on $w.$  The coefficient of $w \cdot \frac{a^{(k)}}{k!}\alpha^{q-1}d \alpha $ in this expression is $\sum _{i+j=q-k}( i \frac{b^{(i)}}{i!}\frac{jc^{(j)}}{j!}-j \frac{ib^{(i)}}{i!}\frac{c^{(j)}}{j!})=0.$ The coefficient of $w \cdot \frac{a^{(k)}}{k!} \alpha^q ds$ in the same expression is 
$$
\sum _{i+j=q-k}( i \frac{b^{(i)}}{i!}\frac{c^{(j+1)}}{j!}-j \frac{b^{(i+1)}}{i!}\frac{c^{(j)}}{j!})=\sum_{i+j=q-k+1 \atop {1 \leq i, \, j}} \frac{b^{(i)}}{(i-1)!} \frac{c^{(j)}}{(j-1)!}-\sum_{i+j=q-k+1 \atop {1 \leq i, \, j}} \frac{b^{(i)}}{(i-1)!} \frac{c^{(j)}}{(j-1)!}=0.
$$
Therefore $\Omega_{m,r}(R)$ can be rewritten as 
\begin{eqnarray}\label{simpleformula}
\sum_{i+j+k=q}\frac{\alpha^ka^{(k)}}{k!} \Big( y\frac{\alpha^ib^{(i)}}{i!}\big(\frac{\alpha^j c^{(j)}}{j!}\big)'-z \big(\frac{\alpha^ib^{(i)}}{i!} \big) ' \frac{\alpha^j c^{(j)}}{j!} \Big)ds.
\end{eqnarray}

The coefficient of $\alpha' \alpha^{q-1}ds$ in the above expression is equal to 
$$
\sum _{0 \leq k \leq q-1} \frac{a^{(k)}}{k!}  \sum_{i+j=q-k}\big(  y \frac{b^{(i)}}{i!}\frac{jc^{(j)}}{j!}- z \frac{ib^{(i)}}{i!}\frac{c^{(j)}}{j!} \big)
$$
which agrees with the coefficient of $\alpha' \alpha^{q-1}ds$ in (\ref{antider}).

Fix $i_{0},\, j_0,$ and $k_0$ such that $i_0+j_0+k_0=q.$ Then the coefficient of $y\alpha ^{q} a^{(k_0)}b^{(i_0)}c^{(j_{0}+1)}$ in (\ref{antider}) is equal to 
$
\frac{1}{q}(\frac{1}{(k_0-1)!}\frac{1}{i_0!}\frac{1}{j_0!}+\frac{1}{k_0!}\frac{1}{i_0!}\frac{1}{(j_0-1)!} +\frac{1}{k_0!}\frac{1}{(i_0-1)!}\frac{1}{j_0!})=\frac{1}{k_0!}\frac{1}{i_0!}\frac{1}{j_0!},
$
which is exactly the same as the coefficient of the same term in (\ref{simpleformula}).
By symmetry, we deduce the same statement for the coefficients of $z\alpha ^{q} a^{(k_0)}b^{(i_0+1)}c^{(j_{0})}.$ This finishes the proof of the proposition. 
\end{proof}

\begin{corollary}\label{omega m+1}
Suppose that $\sigma$ and $e^{at^x} \wedge e^{bt^y} \wedge e^{ct^z}$ are as above. If $r=m+1,$ then $
\Omega_{m,r} \big(\frac{\sigma(e^{at^x}\wedge e^{bt^y}\wedge e^{ct^z} )}{e^{at^x}\wedge e^{bt^y}\wedge e^{ct^z}}\big)=0.
$ 
 \end{corollary}
 
 \begin{proof}
 In this case in order to have $m \leq x$ and $(m+1)-(x+y+z)=r-(x+y+z)>0,$ we have to have $x=m$ and $y=z=0.$ In this case, (\ref{antider}) is equal to 0. 
 \end{proof}

\begin{corollary} If $\pazocal{R}/k_{m+1}$ is a smooth $k_{m+1}$-algebra of relative dimension 1 as above, then for $r=m+1,$ we have a well-defined map  
$$
\omega_{m,m+1}:\Lambda^3 (\pazocal{R},(t^m) )^{\times} \to \Omega^{1} _{\underline{\pazocal{R}}/k}
$$
as in Definition \ref{defn omega path}, which does not depend on the choice of a splitting of $\pazocal{R}/(t^m).$  \end{corollary}

\begin{proof}
This follows immediately from Corollary \ref{omega m+1}, by reducing to the  case $\pazocal{R}=k'((s))_{m+1},$ after localising and completing. 
\end{proof}
For a general $r$ between $m$ and $2m,$ the following corollary will be essential. 

\begin{corollary}\label{corollary indep res}
Fix $m<r<2m,$ and let $\pazocal{R}/k_r$ be a smooth algebra of relative dimension 1 as above. Let $x$ be a closed point of the spectrum of $\pazocal{R},$  $k'$ its residue field, and let $\eta$ be the generic point of $\pazocal{R}.$ Then for any two splittings $\sigma$ and $\sigma'$ of $\pazocal{R}_{\eta}/(t^m),$ the reduction modulo $(t^m)$  of the local ring of $\pazocal{R}$ at $\eta,$ and for any $\alpha \in \Lambda ^3 (\pazocal{R}_{\eta},(t^m))^{\times},$  the residues of  $\omega_{m,r,\sigma}(\alpha)$ and  $\omega_{m,r,\sigma'}(\alpha) \in \Omega^{1}_{\underline{\pazocal{R}}_{\eta}/k}$ at $x$ are the same: 
$$
res_x\omega_{m,r,\sigma'}(\alpha)=res_x\omega_{m,r,\sigma}(\alpha) \in k'.
$$
\end{corollary}

\begin{proof}
Again by localising and completing we reduce to the case of $k'((s))_{r}.$ By Proposition  \ref{main comp residue}, we see that the difference $\omega_{m,r,\sigma'}(\alpha)-\omega_{m,r,\sigma}(\alpha)$ is the differential of an element in $k'((s))$ and hence has zero residue.    
\end{proof}

\begin{remark}
Let $\pazocal{R}/k_r$
 be as above. Suppose that $\tau$ and $\sigma$ are two splittings $\underline{\pazocal{R}}_{m}\to \pazocal{R}/(t^m).$ In this case, there should be  a map 
$$
h\omega_{m,r}(\tau,\sigma):\Lambda^3 (\pazocal{R},(t^m) )^{\times} \to \underline{\pazocal{R}}
$$
such that 
$$
d(h\omega_{m,r}(\tau,\sigma))=\omega_{m,r,\tau}-\omega_{m,r,\sigma}. 
$$
Moreover, $h\omega_{m,r}(\tau,\sigma)$ should vanish on the image of $B_{2}(\pazocal{R},(t^m))\otimes (\pazocal{R},(t^m))^{\times}.$

In case $r=m+1,$ $h\omega_{m,m+1}=0$ does satisfy the properties above.  Let us look at the first non-trivial case when $m=3$ and $r=5.$ Note that the reduction modulo $(t^2)$ of the automorphism $\tau^{-1} \circ \sigma: \underline{\pazocal{R}}_3 \to \underline{\pazocal{R}}_3,$ which lifts the identity map on $\underline{\pazocal{R}},$ is determined by a $k$-derivation $\theta: \underline{\pazocal{R}} \to \underline{\pazocal{R}}.$

Define $h\Omega_{3,5}(\theta):I_{3,5} \subseteq (\Lambda ^{3}\underline{\pazocal{R}}_{5} ^{\times})^{\circ} \to \underline{\pazocal{R}},$ as 
$$
h\Omega_{3,5}(\theta)(e^{at^3}\wedge e^{bt}\wedge c)=ab\theta(\frac{dc}{c}),
$$
where $a,\, b, \in \underline{\pazocal{R}}$ and $c \in \underline{\pazocal{R}}^{\times}.$ Let $h\Omega_{3,5}(\theta)$ be defined as 0 on all the other type of  elements in $I_{3.5}$. Then $h\omega_{3,5}(\tau,\sigma ): \Lambda^3 (\pazocal{R},(t^3) )^{\times} \to \underline{\pazocal{R}} $ defined by 
$$
h\omega_{3,5}(\tau,\sigma )(\alpha) :=-h\Omega_{3,5}(\theta)(s(\sigma^{-1}(\alpha))) 
$$
satisfies the desired properties above. An analog of this construction is one of the main tools in defining an infinitesimal version of the Bloch regulator in \cite{unv4}. 
\end{remark}

\begin{definition}\label{defn residue 1-form}
Let $\pazocal{R}/k_r$ be a smooth algebra of relative dimension 1 as above. Let $\eta$ be the generic point and $x$ be a closed point of the spectrum of $\pazocal{R}.$ Then we have a canonical map 
$$
res_{x}\omega_{m,r}:\Lambda ^3 (\pazocal{R}_{\eta},(t^m))^{\times}\to k',
$$
where $k'$ is the residue field of $x.$ The map is defined by choosing any splitting $\sigma$ of $\pazocal{R}_{\eta}/(t^m)$ and letting $res_{x}\omega_{m,r}:=res_{x}\omega_{m,r,\sigma}.$ This is independent of the choice of the splitting $\sigma$, by Corollary \ref{corollary indep res}.
\end{definition}

\subsection{Variant of the residue map for different liftings.}\label{subsection resdiue of different liftings}  For the construction of the infinitesimal Chow dilogarithm, we need a variant of Definition \ref{defn residue 1-form}. Fortunately, we do not need to do extra work, Corollary \ref{cor semi path}  and Proposition \ref{main comp residue}   will still be sufficient to give us what we are looking for. 

Suppose that  $A$ is a ring with an ideal $I$ and $B$ and $B'$ are two $A$-algebras together with an isomorphism $\chi: B/IB \simeq B'/IB'$ of $A$-algebras. We let 
$$
(B,B',\chi)^{\times}:=\{(p,p')| p \in B^{\times} \;{\rm and} \;p' \in B'^{\times} \; {\rm s.t.}\; \chi(p|_I)=p'|_I \},
$$
where $p|_{I}$ denotes the image of $p$ in $(B/IB)^{\times}.$ Similarly, we define $(B,B',\chi) ^{\flat}$ and $B_{2}(B,B',\chi)$ and obtain maps, $B_{2}(B,B',\chi) \to \Lambda ^{2}(B,B',\chi)^{\times}$ and $B_{2}(B,B',\chi)\otimes (B,B',\chi)^{\times} \to \Lambda ^{3} (B,B',\chi)^{\times}.$ We will use these definitions below with $A=k_{\infty}$ and $I=(t^m).$ In fact the following variant will be essential in what follows. 

Suppose that $\pazocal{S}/k_m$ is a smooth algebra of relative dimension 1, with $x$ a closed point and $\eta$ the generic point of its spectrum.    Suppose that $\pazocal{R},\, \pazocal{R}'/k_r$ are liftings of $\pazocal{S}_{\eta}$ to $k_r.$ In other words, we have  fixed isomorphisms: 
$$
\psi:\pazocal{R}/(t^m)\to \pazocal{S}_{\eta}
$$
and 
$$
\psi':\pazocal{R}'/(t^m)\to \pazocal{S}_{\eta}.
$$

Letting $\chi:={\psi'}^{-1}\circ\psi,$  we would like to construct a map 
$$
res_{x}\omega_{m,r}:\Lambda ^3 (\pazocal{R},\pazocal{R}',\chi)^{\times} \to k',
$$
where $k'$ is the residue field of $x.$ Note that $(\pazocal{R},\pazocal{R}',\chi)^{\times}$ consists of pairs of $(p,p')$ with $p \in \pazocal{R}^{\times}$ and $p' \in {\pazocal{R}'}^{\times}$ such that $\psi(p|_{t^m})=\psi'(p'|_{t^m}).$ In other words, it consists of different liftings of elements of $\pazocal{S}_{\eta} ^{\times}.$ We sometimes use the notation $(\pazocal{R},\pazocal{R}',\psi,\psi')^{\times}$ to denote the same set. 

In order to construct this map,   let $$
\tilde{\chi}:\pazocal{R} \to \pazocal{R}'
$$
be an isomorphism of $k_r$-algebras which is a lifting of  $\chi.$ This provides us with a map 
$$
\xymatrix{
(\pazocal{R},\pazocal{R}',\chi)^{\times}  \ar^{\tilde{\chi}^{*}}[r] & (\pazocal{R},(t^m))^{\times}}.
$$
Choosing a splitting $\sigma: \underline{\pazocal{R}}_{m} \to \pazocal{R}/(t^m),$ by Definition \ref{defn omega path}  we obtain the map $\omega_{m,r,\sigma},$ composing this with the map induced by the reduction $\underline{\psi}$ of $\psi ,$ we obtain 
\begin{align}\label{residue general diagram}
\xymatrix@R+1pc@C+1pc{
 \Lambda^3 (\pazocal{R},\pazocal{R}',\chi)^{\times} \ar^{\Lambda^3\tilde{\chi}^{*}}[r] & \Lambda^3 (\pazocal{R},(t^m))^{\times} \ar[r]^{\;\;\;\omega_{m,r,\sigma}} &\Omega^1 _{\underline{\pazocal{R}}/k} \ar[r]^{d\underline{\psi}}& \Omega^1 _{\underline{\pazocal{S}}_{\eta}/k} \ar^{res_x}[r]& k'.}
\end{align}

\begin{proposition}
The map (\ref{residue general diagram}) above is independent of the choices of the lifting $\tilde{\chi}$ of $\chi$ and the choice of the splitting $\sigma$ of $\pazocal{R}/(t^m).$ 
\end{proposition}

\begin{proof}
That the composition is independent of the choice of $\tilde{\chi}$ follows from Corollary \ref{cor semi path} and Definition \ref{defn omega path}. That it is independent of the choice of the splitting $\sigma$ follows from Proposition \ref{main comp residue}. 
\end{proof}

\begin{definition}
We denote the composition (\ref{residue general diagram}) above by 
$$
res_{x}\omega_{m,r}(\psi,\psi'):  \Lambda^3 (\pazocal{R},\pazocal{R}',\psi,\psi')^{\times} \to k'.
$$
If $\psi$ and $\psi'$ are clear from the context, we denote this map by $res_x\omega_{m,r},$ and $(\pazocal{R},\pazocal{R}',\psi, \psi')^{\times}$ by $(\pazocal{R},\pazocal{R}',(t^m))^{\times}.$ Depending on the context, we also use the notation $res_x\omega_{m,r}(\chi):\Lambda ^3(\pazocal{R},\pazocal{R}',\chi) \to k'$ for the same map, with $\chi=\psi'^{-1}\circ\psi.$
\end{definition}
With these definitions, we have the following corollary.

\begin{corollary}\label{corollary value of omega on a boundary}
Suppose that $\pazocal{R}$ and $\pazocal{R}'$ are smooth $k_r$-algebras of dimension 1 as above which are liftings of the generic local ring $\pazocal{S}_{\eta}$ of a smooth $k_m$-algebra $\pazocal{S}.$ Let  $\chi: \pazocal{R}/(t^m) \to \pazocal{R}'/(t^m)$ be the corresponding isomorphism of $k_m$-algebras. Let $x$ be a closed point of $\underline{\pazocal{S}}.$ Then the map 
$$
res_{x}\omega_{m,r}(\chi):\Lambda ^3(\pazocal{R},\pazocal{R}',\chi)^{\times}\to k'
$$
vanishes on the image of $B_{2}(\pazocal{R},\pazocal{R}',\chi)\otimes (\pazocal{R},\pazocal{R}',\chi)^{\times}.$ 

\end{corollary}

\begin{proof}
Follows from Proposition \ref{value of omega on a boundary}. 
\end{proof}

\section{The residue of $\omega_{m,r}$ on good liftings.} 

Suppose that $\pazocal{R}/k_r$ is as above. Moreover, we assume   that the reduction $\underline{\pazocal{R}}$ of $\pazocal{R}$ modulo $(t)$ is a discrete valuation ring with $x$ being the closed point. We let $\tilde{x}/k_r$ be a {\it lifting} of $x$ to $\pazocal{R}.$  By this what we mean is as follows. Let $s$ be a uniformizer at $x,$ and let $\tilde{s}$ be any lifting of $s$ to $\pazocal{R},$ we call $\tilde{s}$ also a uniformizer at $x$ on $\pazocal{R}.$ The associated scheme $\tilde{x},$ which is smooth over $k_r,$ is what we call a lifting of $x.$ In other words a lifting of $x$ is a 0-dimensional closed subscheme $\tilde{x}$ of $\pazocal{R}$ such that its ideal is generated by a single element which reduces to a uniformizer on the closed fiber. Note that if we are given $\tilde{x},$ then $\tilde{s}$ is determined up to a unit in $\pazocal{R}.$ Sometimes we will abuse the notation and write $(\tilde{s})$ instead of $\tilde{x}.$   Let $\eta$ denote the generic point of $\pazocal{R}.$ We let 
$$
(\pazocal{R},\tilde{x})^{\times}:=\{\alpha \in \pazocal{R} _{\eta} ^{\times}| \alpha=u\tilde{s}^n,\; {\rm for\; some} \;\; u \in \pazocal{R}^{\times}\; {\rm and} \; n \in \mathbb{Z}  \}.
$$
We  say that  an element $\alpha \in \pazocal{R}^{
\times} _{\eta}$ is  {\it good with respect to } $\tilde{x},$ if $\alpha \in (\pazocal{R},\tilde{x})^{\times}.$  Note that this property depends only on $\tilde{x},$ and not on $\tilde{s}.$ The importance of this notion for us is that for wedge products of good liftings, we can define their residue along $(s)$ as in  \cite[\textsection 2.4.5]{unv3}. Namely, there is a map 
$$
res_{\tilde{x}}: \Lambda ^{n}(\pazocal{R},\tilde{x})^{\times} \to \Lambda^{n-1}(\pazocal{R}/(s))^{\times},
$$
with the properties that it vanishes on $\Lambda ^{n}\pazocal{R}^{\times}$ and $s\wedge \alpha_{1}\wedge \cdots \wedge \alpha_{n-1}$ is mapped to $\underline{\alpha}_{1}\wedge \cdots \wedge \underline{\alpha}_{n-1},$ if $\alpha_ i \in \pazocal{R}^{\times}$ and $\underline{\alpha}_{i}$
 denotes the image of $\alpha _i$ in $(R/(s))^{\times},$ for $1 \leq i \leq n-1.$  

Suppose that $\pazocal{R}'/k_r$  is another such ring, and $\tilde{x}'$ a lifting of the closed point of $\underline{\pazocal{R}}'.$  Suppose that there is an isomorphism $\chi: \pazocal{R}/(t^m) \to \pazocal{R}'/(t^m)$ which transfers $\tilde{x}/(t^m)$ to $\tilde{x}'/(t^m).$ Then we let 
$$
(\pazocal{R},\pazocal{R}',\tilde{x},\tilde{x}',\chi)^{\times}:=\{(p,p')| p \in (\pazocal{R},\tilde{x})^{\times} \;{\rm and} \;p' \in (\pazocal{R}',\tilde{x}')^{\times} \; {\rm s.t.}\; \chi(p|_{t^m})=p'|_{t^m} \}.
$$
Note that clearly $(\pazocal{R},\pazocal{R}',\tilde{x},\tilde{x}',\chi)^{\times}\subseteq (\pazocal{R},\pazocal{R}',\chi)^{\times}.$ In case $\pazocal{R}'=\pazocal{R}$ with $\chi$ the identity map, we denote the corresponding group by $(\pazocal{R},\tilde{x},(t^m))^{\times}.$ Denote the natural maps $(\pazocal{R},\pazocal{R}',\tilde{x},\tilde{x}',\chi)^{\times} \to (\pazocal{R},\tilde{x})^{\times}$ and 
  $(\pazocal{R},\pazocal{R}',\tilde{x},\tilde{x}',\chi)^{\times} \to (\pazocal{R}',\tilde{x}')^{\times}$ by $\pi_1$ and $\pi_2.$ 
 
In this section, we would like to compute $res_x \omega_{m,r}(\chi)(\alpha) $ for $\alpha \in \Lambda ^3(\pazocal{R},\pazocal{R}',\tilde{x},\tilde{x}',\chi)^{\times}$ in terms of the value of $\ell  _{m,r}$ on the residue of $\alpha.$ The main result of this section is Proposition \ref{general residue prop}. We will first start with certain explicit computations on the formal power series rings and then finally  reduce our general statement to these special cases. Let us immediately remark that in order to compute the residues, we immediately reduce to the case when $\pazocal{R}$ and $\pazocal{R}'$ are complete with respect to the ideal which correspond to their closed points.

We will first consider the case of $R=k'[[s]]$ and that of the same uniformizer on both of the liftings as follows. 

 Note that 
$$
res_{(s)}(s\wedge \alpha \wedge \beta)=\underline{\alpha}\wedge \underline {\beta} \in \Lambda ^2 {k'_r} ^{\times},
$$
where $\underline{\alpha}$ and $\underline {\beta}$ are the images of $\alpha$ and $\beta$ under the natural projection $\pazocal{R}^{\times} \to (\pazocal{R}/(s))^{\times}={k'_r} ^{\times}. $ 
Similarly for $p'.$ 

\begin{lemma}\label{lemma same s}
 Suppose that $R=k'[[s]],$   and $\pazocal{R}:=R_r=k'[[s,t]]/(t^r).$   
Suppose that $\alpha, \alpha' ,$ and $\beta \in \pazocal{R} ^{\times}$ such that $\alpha'|_{t^m} =\alpha |_{t^m} \in \pazocal{R}/(t^m).$ Let $p':=s\wedge \alpha' \wedge \beta,$  $p:=s\wedge \alpha \wedge \beta$ and  $(p,p'):=(s,s) \wedge (\alpha,\alpha')\wedge (\beta,\beta) \in \Lambda^3(\pazocal{R},(s),(t^m))^{\times}\subseteq \Lambda ^3 (\pazocal{R}_{\eta},(t^m))^{\times}.$ Then  the residue of $ \omega_{m,r}(p,p' ) $ at the closed point of $\underline{\pazocal{R}}=k'[[s]]$ is given by 
$$
res_{s=0} \omega_{m,r}(p,p' )=\ell_{m,r}(res_{(s)}(p))-\ell_{m,r}(res_{(s)}(p')).
$$
\end{lemma}

\begin{proof}
By  Definition \ref{defn om}, we see that $\omega_{m,r}(p,p')=\Omega_{m,r}(p-p')=\Omega_{m,r}(s\wedge \frac{\alpha}{\alpha'}\wedge \beta).$ Let us compute the residue at $s=0$ of an expression of the type
$
\Omega_{m,r}(s\wedge e^{at^i}\wedge e^{bt^j}),
$
with $a, b \in k[[s]]$ and $i\geq m ,$  such that if $j=0,$ we use the convention in Proposition \ref{formula Omega}.  Since $$
\Omega_{m,r}(s\wedge e^{at^i}\wedge e^{bt^j})=jab\frac{ds}{s},
$$
when $i+j=r$ and is 0 otherwise, 
 by Proposition \ref{formula Omega}, we conclude that    its residue is equal to  $ja(0)b(0)$ if $i+j=r,$ and is 0 otherwise. Since, for $i \geq m,$  
 $$
 \ell_{m,r}(res_{(s)} (s\wedge e^{at^i}\wedge e^{bt^j}))=\ell _{m,r}(e^{a(0)t^i}\wedge e^{b(0)t^j})
 $$
 is equal to $ja(0)b(0)$ if $i+j=r,$ and is 0 otherwise, we conclude that 
 \begin{align}\label{eqn res om loc}
    res_{s=0}(\Omega_{m,r}(s\wedge e^{at^i}\wedge e^{bt^j}))= \ell_{m,r}(res_{(s)} (s\wedge e^{at^i}\wedge e^{bt^j})). 
 \end{align}
On the other hand, since $\alpha'|_{t^m}=\alpha|_{t^m},$ $s\wedge \frac{\alpha}{\alpha'}\wedge \beta$ is a sum of terms of the above type, and the linearity of both sides of (\ref{eqn res om loc}) imply that (\ref{eqn res om loc}) is also valid for $s\wedge \frac{\alpha}{\alpha'}\wedge \beta.$ By linearity of $\ell _{m,r}$ and $res_{(s)},$ we have  
$$
\ell_{m,r}(res_{(s)} (s\wedge \frac{\alpha}{\alpha'}\wedge \beta ))=\ell_{m,r}(res_{(s)} (s\wedge \alpha\wedge \beta))-\ell_{m,r}(res_{(s)} (s\wedge \alpha'\wedge \beta ) ),
$$
 which together with the above proves  the lemma. 
 \end{proof}

 Let us now try to prove the same formula when the choice of the uniformizer is not the same. In other words, with notation as above let $s' \in \pazocal{R}$ such that $s'|_{t^m}=s|_{t^m}.$ For simplicity, let us temporarily use the notation $(\pazocal{R},(s),(s'),(t^m))^{\times}:=(\pazocal{R},\pazocal{R},(s),(s'),id_{\pazocal{R}/(t^m)})^{\times}.$ Let $p':=s'\wedge\alpha \wedge \beta, $   $p:=s\wedge \alpha \wedge \beta $ and $(p,p'):=(s,s') \wedge (\alpha,\alpha)\wedge (\beta,\beta) \in \Lambda ^3 (\pazocal{R},(s),(s'),(t^m))^{\times}.$
  
  \begin{lemma}\label{lemma diff s}
   With notation as above, the residue of $\omega_{m,r}(p,p')$ at the closed point of $k'[[s]]$ is given by the following formula: 
   $$
   res_{s=0} \omega_{m,r}(p,p') = \ell_{m,r}(res_{(s)}(p))-\ell_{m,r}(res_{(s')}(p')).
   $$
  \end{lemma}

 \begin{proof}
If $s'' $ is another lift of the uniformizer $s,$ in other words $s'' \in \pazocal{R}$ with $s''|_{t^m}=s|_{t^m}$ then 
$$
   res_{s=0} \omega_{m,r}(p,p'')=   res_{s=0} \omega_{m,r}(p,p')+   res_{s=0} \omega_{m,r}(p',p'')
$$
and $ \ell_{m,r}(res_{(s)}(p))-\ell_{m,r}(res_{(s'')}(p''))=$
$$
\big(  \ell_{m,r}(res_{(s)}(p))-\ell_{m,r}(res_{(s')}(p'))\big)+\big( \ell_{m,r}(res_{(s')}(p'))-\ell_{m,r}(res_{(s'')}(p''))\big).
$$
Therefore in order to prove the lemma we may assume without loss of generality that $s'=s+at^i,$ with $a \in k'[[s]]$ and $m \leq i.$ Note that in $\pazocal{R},$ we have $s+at^i=se^{\frac{a}{s}t^i},$ since $r<2m.$ Letting $\alpha=e^{bt^j}$ and $\beta=e^{ct^k},$ we can rewrite $\omega_{m,r}(p,p')$ as $\Omega_{m,r}(p-p')=\Omega_{m,r}(e^{-\frac{a}{s}t^i}\wedge e^{bt^j} \wedge e^{ct^k})=$
$$
-\frac{a}{s}(jb\cdot dc-kc\cdot db)
$$
by Proposition \ref{formula Omega}, if $i+j+k=r$ and  0 otherwise. Its residue is 
\begin{align}\label{diff of residues formula in lemma}
-a(0)(jb(0)c'(0)-kc(0)b'(0))
\end{align}
if $i+j+k=r $ and 0 otherwise, with the usual conventions if $j$ or $k$ is 0.

On the other hand, $res_{(s)}(p)=e^{b(0)t^j}\wedge e^{c(0)t^k}$ and 
$$
res_{(s')}(p')=e^{b(0)t^j-a(0)b'(0)t^{i+j}}\wedge e^{c(0)t^k-a(0)c'(0)t^{i+k}}\in \Lambda ^2{k'_r}^{\times}.
$$
By the linearity of $\ell _{m,r},$ the right hand side of the expression in the statement of the lemma is then equal to 
$$
-\ell _{m,r}(e^{b(0)t^j}\wedge e^{-a(0)c'(0)t^{i+k}})-\ell_{m,r}(e^{-a(0)b'(0)t^{i+j}}\wedge e^{c(0)t^k})-\ell_{m,r}(e^{-a(0)b'(0)t^{i+j}}\wedge e^{-a(0)c'(0)t^{i+k}}).
$$
The last summand is equal to 0 since $\ell _{m,r}$ is of weight $r$ and $i+j+i+k\geq 2i\geq 2m>r.$ For the same reason, the first two summands are 0 if $i+j+k \neq r$
and if $i+j+k=r,$ then the total expression is equal to 
$
-a(0)c'(0)jb(0)+a(0)b'(0)kc(0),
$
which agrees with the formula (\ref{diff of residues formula in lemma}) for the residue of $\Omega_{m,r}$. 
Since $\alpha $ and $\beta $ are sums of the terms of the above type, this proves the lemma. 
\end{proof}

 \begin{proposition}\label{general residue prop}
 Suppose that $\pazocal{R},\, \pazocal{R}'/k_r$ are local algebras which are smooth of relative dimension 1 over $k_r$,  together with liftings $\tilde{x}, \, \tilde{x}'$ of their closed points and a $k_{m}$-isomorphism $\chi:\pazocal{R}/(t^m) \to \pazocal{R}'/(t^m)$ which maps the reductions of $\tilde{x}$ and $\tilde{x}'$ to each other. Then for $q \in \Lambda^3(\pazocal{R},\pazocal{R}',\tilde{x},\tilde{x}',\chi)^{\times},$ we have the following formula for the at the closed point $x,$
 $$
res_{x}\omega_{m,r}(q)=\ell _{m,r}(res_{\tilde{x}} (\Lambda ^3\pi_1) (q))-\ell _{m,r}(res_{\tilde{x}'} (\Lambda ^3\pi_2) (q)).
 $$
 \end{proposition}

 \begin{proof}
 In order to prove the statement, we can replace $\pazocal{R}$ and $\pazocal{R}'$ with their completions at their closed points.  Therefore, without loss of generality, we will assume that $\pazocal{R}=\pazocal{R}'=k'[[s]]_r,$ $\tilde{x}$  is given by $s=0$ and $\tilde{x}'$ is given by $s'=0$ for  some $s' \in k'[[s]]_r$ with $s'|_{t^m}=s|_{t^m},$ and $\chi$ is the map which is identity on $k'[[s]]_{m}.$ 
 
 In order to make the computations we need to choose a lifting $\tilde{\chi}$ of $\chi$ from $\pazocal{R}$ to $\pazocal{R}'.$ We choose this lifting to be the one that sends $s$ to $s'$ and is identity on $k'.$ Note that $\tilde{\chi}$ being a map of $k_r$ algebras has to satisfy $\tilde{\chi}(t)=t.$  
 
 The statement above then reduces to the following: suppose that $\alpha, \beta, \, \gamma \in (k'[[s]]_r,(s))^{\times}$ and $\alpha', \beta', \, \gamma' \in (k'[[s]]_r,(s'))^{\times}$ such that $\alpha|_{t^m}=\alpha'|_{t^m},$ $\beta|_{t^m}=\beta'|_{t^m},$ and $\gamma|_{t^m}=\gamma'|_{t^m},$ and $p=\alpha \wedge \beta \wedge \gamma,$ $p'=\alpha' \wedge \beta' \wedge \gamma',$ and $(p,p')=(\alpha,\alpha') \wedge (\beta, \beta') \wedge (\gamma,\gamma'),$ then 
 $$
 res_{s=0}\Omega_{m,r}(p-p')=\ell _{m,r}(res_{(s)}(p))-\ell _{m,r}(res_{(s')}(p')).
 $$
 By assumption  $\alpha$ is of the form $us^{n}$  for some $u \in k'[[s]]_{r} ^{\times}$ and $n \in \mathbb{Z}.$ Similarly, $\alpha'$ is of the form $u's'^{n'},$ with $u' \in k'[[s]]_r ^{\times}.$ The condition that $\alpha|_{t^m}=\alpha'|_{t^m}$ implies that $n=n'.$ The same is true for $\beta ,\,\beta',$ and $\gamma, \, \gamma'.$ By multi-linearity and anti-symmetry, we reduce to checking the above identity in the following two cases: in the first case where $\alpha, \, \beta, \gamma \in k'[[s]]^{\times}$ and  in the second case where $\alpha=s,$ $\alpha'=s'$ and $\beta, \, \gamma \in k'[[s]]^{\times}.$  

If $\alpha, \, \beta, \gamma \in k'[[s]]^{\times},$ then $\alpha', \, \beta', \gamma' \in k'[[s]]^{\times}.$ This implies on the one hand that $res_{(s)}(p)=0$ and $res_{(s')}(p')=0,$ and on the other that $p-p' \in I_{m,r}=(1+(t^m)\otimes \Lambda ^2k'[[s]]_r ^{\times})\subseteq (\Lambda ^{3}k'[[s]]_r ^{\times})^{\circ},$  which implies  that $\Omega_{m,r}(p-p') \in \Omega^1 _{k'[[s]]/k}.$ Therefore $res_{s=0}\Omega_{m,r}(p-p')=0=\ell _{m,r}(res_{(s)}(p))-\ell _{m,r}(res_{(s')}(p'))$ in this case. 
 
Let us now consider the more interesting case of $\alpha=s,$ $\alpha'=s'$ and $\beta,\, \gamma, \,\beta', \,\gamma' \in k'[[s]]^{\times},$ with $\beta|_{t^m}=\beta'|_{t^m}$ and $\gamma|_{t^m}=\gamma'|_{t^m}.$  Applying Lemma \ref{lemma same s} first with $p=(s,\alpha,\beta)$ and  $p'=(s,\alpha',\beta)$ then with $p=(s,\alpha',\beta)$ and  $p'=(s,\alpha',\beta')$  and then applying Lemma \ref{lemma diff s} with $p=(s,\alpha',\beta')$  and $p'=(s',\alpha',\beta')$ and adding all the equalities finishes the proof of the proposition. 
 \end{proof}

\section{Construction of $\rho$ and a regulator on curves}

\subsection{Regulators on curves. } 
Let $\pazocal{R}/k_m$ be  smooth of relative dimension 1, as in the previous section but without the  assumption that $\underline{\pazocal{R}}$ is a discrete valuation ring.   Choose and fix  a lifting $\mathfrak{c}$ of $c$ to $\pazocal{R}$ for every closed point $c$ of $\underline{\pazocal{R}}$ as in the previous section.  We denote  the set of these liftings by $\mathcal{P}.$ We let $k(c)$ denote the residue field of $c$  and  $k(\mathfrak{c})$ denote the artin ring of regular functions on $\mathfrak{c}.$ Let $|\underline{\pazocal{R}}|=|\pazocal{R}|$ denote the set of closed points of $\underline {\pazocal{R}},$ or equivalently of $\pazocal{R}.$    Note that the reductions of the localizations $\pazocal{R}_{c}$ of $\pazocal{R}$ are discrete valuation rings. We let $$(\pazocal{R},\mathcal{P})^{\times}:=\bigcap_{c \in |\pazocal{R}| }(\pazocal{R}_{c },\mathfrak{c})^{\times}$$
and 
$(\pazocal{R},\mathcal{P})^{\flat}:=\{f \in (\pazocal{R},\mathcal{P})^{\times}|1-f \in (\pazocal{R},\mathcal{P})^{\times}  \}.$  We define $B_{2}(\pazocal{R},\mathcal{P})$ to be the vector space over $\mathbb{Q}$ generated by the symbols $[f]$ with $f \in (\pazocal{R},\mathcal{P})^{\flat}$ modulo the five term relations associated to pairs $f$ and $g$ in $(\pazocal{R},\mathcal{P})^{\flat}$ which have the property that $f-g\in (\pazocal{R},\mathcal{P})^{\times}.$  As usual we have  maps $B_{2}(\pazocal{R},\mathcal{P}) \to \Lambda ^2(\pazocal{R},\mathcal{P})^{\times} $ and $  B_{2}(\pazocal{R},\mathcal{P})\otimes(\pazocal{R},\mathcal{P})^{\times}\to \Lambda ^3(\pazocal{R},\mathcal{P})^{\times}.$ We also have a residue map  $res_{\mathfrak{c}}:B_{2}(\pazocal{R},\mathcal{P}) \otimes (\pazocal{R},\mathcal{P})^{\times}\to B_{2}(k(\mathfrak{c}))$ that is  defined exactly as in \cite[\textsection 3.3.1]{unv3} and  which gives a commutative diagram: 
$$
\xymatrix{
B_{2}(\pazocal{R},\mathcal{P}) \otimes (\pazocal{R},\mathcal{P})^{\times}\ar^{res_{\mathfrak{c}}}[d] \ar^{\;\;\;\;\;\;\;\delta}[r] & \Lambda ^3(\pazocal{R},\mathcal{P})^{\times} \ar^{res_{\mathfrak{c}}}[d] \\
B_{2}(k(\mathfrak{c})) \ar^{\delta}[r] & \Lambda ^{2}k(\mathfrak{c})^{\times}.
 }
$$

Suppose  that $C/k_m$ is a  smooth and projective curve. For every closed point $c $ of $\underline{C},$ choose and fix a smooth lifting of $\mathfrak{c}$ of  $c$ to $C.$ We denote $\mathcal{P}$ to be the set of these liftings.   We let $(\pazocal{O}_{C},\underline{\mathcal{P}}) ^{\times}$ denote the sheaf  on $\underline{C}$ which associates to an open set $U$ of $\underline {C},$ the group $(\pazocal{O}_{C}(U),\mathcal{P}|_{U})^{\times}.$ Similarly, $B_{2}(\pazocal{O}_C, \underline{\mathcal{P}})$ is the sheaf associated to the presheaf, which associates to $U$ the group $B_{2}(\pazocal{O}_{C}(U),\pazocal{P}|_{U}).$ For each $c \in |C|,$ let   $i_{c}$ denote  the imbedding of $c$ in $\underline{C}.$ The commutative diagram above gives us a  complex $\underline{\Gamma}'_{B}(C,\mathcal{P},3)$ of sheaves: 
$$
B_{2}(\pazocal{O}_C, \underline{\mathcal{P}}) \otimes (\pazocal{O}_{C},\underline{\mathcal{P}}) ^{\times}\to \oplus _{c \in |C|}i_{c*}(B_{2}(k(\mathfrak{c}))) \oplus \Lambda ^3 (\pazocal{O}_{C},\underline{\mathcal{P}}) ^{\times} \to  \oplus _{c \in |C|}i_{c*}( \Lambda ^2 k(\mathfrak{c}) ^{\times}),
$$
concentrated in degrees $[2,4].$ We use the following sign conventions in the above complex: the first map is $(\delta, res)$ and the second one is  $-\delta+res.$ We will be interested in the infinitesimal part of the  degree 3 cohomology ${\rm H}^{3} _{B}(C,\mathbb{Q}(3) ):={\rm H}^{3}(\underline{C},\underline{\Gamma}'_{B}(C,\mathcal{P},3) )$ of the complex $\underline{\Gamma}'_{B}(C,\mathcal{P},3).$ More precisely, we will be interested in defining regulator maps from ${\rm H}^{3} _{B}(C,\mathbb{Q}(3) )$ to $k$ for every $m<r<2m.$ 
 
The above cohomology group is a candidate for the motivic cohomology group ${\rm H}^{3}_{\pazocal{M}}(C,\mathbb{Q}(3)).$ To be more precise, we would expect a sheaf $B_{3}(\pazocal{O}_C, \underline{\mathcal{P}}) $ of Bloch groups of weight 3 as in \cite{config},
 which would fit into a complex   
$\underline{\Gamma}_{B}(C,\mathcal{P},3)$ of sheaves on $\underline{C}$: 
$$
B_{3}(\pazocal{O}_C, \underline{\mathcal{P}}) \to   B_{2}(\pazocal{O}_C, \underline{\mathcal{P}})  \otimes (\pazocal{O}_{C},\underline{\mathcal{P}}) ^{\times}\to \oplus i_{c*}(B_{2}(k(\mathfrak{c}))) \oplus \Lambda ^3 (\pazocal{O}_{C},\underline{\mathcal{P}}) ^{\times} \to  \oplus i_{c*}( \Lambda ^2 k(\mathfrak{c}) ^{\times}),
$$ 
and which would compute motivic cohomology of weight 3. 
Since we are only interested in  ${\rm H}^3(\underline{C},\underline{\Gamma}_{B}(C,\mathcal{P},3))$ and since on a curve, by Grothendieck's vanishing theorem, the  cohomology of any sheaf vanishes in degree greater than 1, we have an isomorphism  
$$
{\rm H}^3(\underline{C},\underline{\Gamma}'_{B}(C,\mathcal{P},3)) \simeq {\rm H}^3(\underline{C},\underline{\Gamma}_{B}(C,\mathcal{P},3)).
$$

For a sheaf of complexes $\mathcal{F}_{\bigcdot},$ let  $\check{\rm H}^{\bigcdot}(\underline{C},\mathcal{F}_{\bigcdot})$ denote the colimit of all the \v{C}ech cohomology groups over all Zariski covers of $\underline{C}.$ For a sheaf $\mathcal{F},$  the natural map $\check{\rm H}^{i}(\underline{C},\mathcal{F}) \to {\rm H}^{i}(\underline{C},\mathcal{F})$ is an isomorphism for $i=0, \, 1.  $ By  the same argument, it follows that the same is true for a complex of sheaves $\mathcal{F}_{\bigcdot}$ which is concentrated in degrees 0 and 1. This applied to the complex above implies that the natural map 
$$
\check{\rm H}^3(\underline{C},\underline{\Gamma}'_{B}(C,\mathcal{P},3))\simeq {\rm H}^3(\underline{C},\underline{\Gamma}'_{B}(C,\mathcal{P},3))
$$
is an isomorphism. Therefore, it is enough to construct the map $\check{\rm H}^3(\underline{C},\underline{\Gamma}'_{B}(C,\mathcal{P},3)) \to k.$ We will in fact construct the map as the composition
$$
\check{\rm H}^3(\underline{C},\underline{\Gamma}'_{B}(C,\mathcal{P},3)) \hookrightarrow \check{\rm H}^3(\underline{C},\underline{\Gamma}''_{B}(C,\mathcal{P},3)) \to k,
$$
where $\underline{\Gamma}''_{B}(C,\mathcal{P},3)$ is the quotient complex:
$$
 B_{2}(\pazocal{O}_C, \underline{\mathcal{P}}) \otimes (\pazocal{O}_{C},\underline{\mathcal{P}}) ^{\times}\to \oplus _{c \in |C|}i_{c*}(B_{2}(k(\mathfrak{c}))) \oplus \Lambda ^3 (\pazocal{O}_{C},\underline{\mathcal{P}}) ^{\times}
$$
of  $\underline{\Gamma}'_{B}(C,\mathcal{P},3).$

Suppose that we are given a Zariski open cover $U_{\bigcdot}$ of $\underline{C}$, we will   define a map from the corresponding cocycle group 
$\check{{\rm Z}} ^{3}(U_{\bigcdot}, \underline{\Gamma}''_{B}(C,\mathcal{P},3))$ to $k,$ which will vanish on the coboundaries and hence induce the map in the cohomology group that we are looking for.  
Suppose that we start with a cocyle as above, given by the data:  

(i) $\gamma_{i} \in \Lambda ^3 (\pazocal{O}_{C},\underline{\mathcal{P}}) ^{\times}(U_i),$ for all $i \in I.$ 

(ii) $\varepsilon_{i,c} \in B_{2}(k(\mathfrak{c})) $ for every $c \in U_{i}$ all but finitely many of which are 0, for all $i \in I$

(iii) $\beta_{ij} \in  (B_{2}(\pazocal{O}_C, \underline{\mathcal{P}}) \otimes (\pazocal{O}_{C},\underline{\mathcal{P}}) ^{\times})(U_{ij}),$ for all  $i, \, j \in I$

These data are supposed to satisfy the following  properties:

(i) $\delta(\beta_{ij})=\gamma_{j}|_{U_{ij}}-\gamma_i|_{U_{ij}},$ 

(ii) $res_{\mathfrak{c}}(\beta_{ij})=\varepsilon_{j,c}-\varepsilon_{i,c},$ for $c \in U_{ij} ,$

(iii) $\beta_{jk}|_{U_{ijk}}-\beta_{ik}|_{U_{ijk}}+\beta_{ij}|_{U_{ijk}}=0.$

We  will construct the image of the above element by making several choices and then proving that the construction is independent of all the choices.

(i) Let $\tilde{\pazocal{A}}_{\eta}/k_{\infty}$ be a smooth lifting of $\pazocal{O}_{C,\eta}$ and for every $c \in |C|,$ let $\tilde{\pazocal{A}}_{c}/k_{\infty}$ be a smooth lifting of the completion $\hat{\pazocal{O}}_{C,c}$ of the local ring of $C$ at $c,$ together with a smooth lifting $\tilde{\mathfrak{c}}$ of $\mathfrak{c}$ as in the previous section. 
Moreover, choose:

(ii) an arbitrary $i \in I$ and for each $c$ choose a $j_{c} \in I$ such that $c \in U_{j_c}$ 

(iii) an arbitrary lifting $\tilde{\gamma}_{i\eta} \in \Lambda^3\tilde{\pazocal{A}}_{\eta} ^{\times}$  of the germ $\gamma _{i\eta} \in \Lambda ^3\pazocal{O}_{C,\eta}^{\times}$  of $\gamma_i$ at the generic point $\eta$ 

(iii) a good lifting $\tilde{\gamma}_{j_c} \in \Lambda ^3(\tilde{\pazocal{A}}_c,\tilde{\mathfrak{c}})^{\times}$  of the image $\hat{\gamma}_{j_c,c}$ of  $\gamma_{j_c}$ in $\Lambda^3(\hat{\pazocal{O}}_{C,c},\mathfrak{c})^{\times},$ for every $c \in |C|,$  

(iv) an arbitrary lifting   $\tilde{\beta}_{j_c i,\eta} \in B_{2}(\tilde{\pazocal{A}}_{\eta})\otimes \tilde{\pazocal{A}}_{\eta}$  of the image  $\beta_{j_c i,\eta}\in B_{2}(\pazocal{O}_{C,\eta})\otimes \pazocal{O}_{C,\eta} ^{\times}$ of   $\beta_{j_c i},$ for every $c \in |C|.$    
 
Note that it does not make sense to require that $\tilde{\gamma}_{i\eta}$  be a good lifting since in this context there is no  a fixed specialization of the generic point. Similarly,  we cannot require that $\tilde{\beta}_{j_c i,\eta}$ be a good lifting,  since we know that $\delta(\tilde{\beta}_{j_c i,\eta})$ is a lifting of  $\delta(\beta_{j_c i,\eta})=\gamma_{i} -\gamma_{j_c}$ and even  this last expression need not be good at $c$ as $\gamma_{i}$ need not be good at $c.$

We define the value of the regulator $\rho_{m,r}$ on the above element by the expression
\begin{align}\label{reg formula in general}
   \sum _{c \in |C|}{\rm Tr}_{k}\big(\ell _{m,r}(res_{\tilde{\mathfrak{c}}}\tilde{\gamma}_{j_c})-\ell i _{m,r} (\varepsilon_{j_c,c})+res_{c}\omega_{m,r}(\tilde{\gamma}_{i\eta}-\delta(\tilde{\beta}_{j_ci,\eta}),\tilde{\gamma}_{j_c})\big). 
\end{align}
Let us first explain what we mean by the above expression. Since $\tilde{\gamma}_{j_c}$ is $\tilde{\mathfrak{c}}$-good, the residue $res_{\tilde{\mathfrak{c}}}\tilde{\gamma}_{j_c}$ along $\tilde{\mathfrak{c}}$  is defined as an element of $\Lambda ^{2}k(\tilde{\mathfrak{c}})^{\times}.$  The  \'{e}taleness of $\tilde{\mathfrak{c}}$ over $k_{\infty},$ implies that we have a canonical isomorphism $k(\tilde{\mathfrak{c}})\simeq k(c)_{\infty}$ of $k_{\infty}$-algebras. Using this isomorphism and the map $\ell _{m,r}:\Lambda ^2 k(c)_{\infty} ^{\times} \to k(c)$ in Definition \ref{defnlmr}, we obtain $\ell _{m,r}(res_{\tilde{\mathfrak{c}}}\tilde{\gamma}_{j_c}) \in k(c).$ For the second term,  note that, as above,  there is a canonical isomorphism $k(\mathfrak{c})\simeq k(c)_{m}$ of $k_m$ algebras using which we can view $\varepsilon_{j_c,c}\in B_{2}(k(c)_{m}).$ Applying $\ell i_{m,r}:B_{2}(k(c)_m)\to k(c)$ to this element gives $\ell i_{m,r}(\varepsilon_{j_c,c})\in k(c).$ For the last term, note that $\tilde{\gamma}_{i\eta}-\delta(\tilde{\beta}_{j_ci,\eta})$ is a lifting of $\gamma_{i\eta}-\delta(\beta_{j_ci,\eta})=\gamma_{j_c}$ to $\Lambda ^{3}\tilde{\pazocal{A}}_{\eta} ^{\times}$ and so is $\tilde{\gamma}_{j_c}$ a lifting of $\gamma_{j_c}$ to $\Lambda ^3\tilde{\pazocal{A}}_{c} ^{\times}.$  Using the theory of \textsection \ref{subsection resdiue of different liftings}, we see that the last term 
$res_{c}\omega_{m,r}(\tilde{\gamma}_{i\eta}-\delta(\tilde{\beta}_{j_ci,\eta}),\tilde{\gamma}_{j_c})\in k(c)$ is unambiguously defined. Letting  ${\rm Tr}_k$ denote the normalized trace to $k,$ the summands above are defined.  

In order to show that the sum makes sense, we also need to show that the sum is finite. Below we will show that the sum is independent of all the choices, therefore it will be enough to show that the sum is finite for a particular choice. 
First by shrinking $U_i$ if necessary, and choosing a refinement of the cover, we will assume that $\gamma_{i} \in \Lambda ^3 \pazocal{O}_{C} ^{\times}(U_i).$ Similarly, by shrinking $U_i$ even further, we will assume that the lifting $\tilde{\gamma}_i$ is good on $U_i.$  Therefore, for $c \in U_i,$ we can choose $j_c=i$ and $\tilde{\gamma}_{j_c}=\tilde{\gamma}_i.$ Since for these c,  $\beta_{j_ci}=0$ we can choose $\tilde{\beta}_{j_ci}=0.$ In order to show that the sum in (\ref{reg formula in general}) is finite, we can concentrate on $c \in U_i,$ as $|C|\setminus|U_i|$ is finite. For $c \in U_i,$ $res_{c}\tilde{\gamma}_{j_c}=res_{c}\tilde{\gamma}_{i}=0,$ since $\gamma_{i} $ is invertible on $U_i$ by assumption. Also for the residues we have  $res_{c}\omega_{m,r}(\tilde{\gamma}_{i\eta}-\delta(\tilde{\beta}_{j_ci,\eta}),\tilde{\gamma}_{j_c})=res_{c}\omega_{m,r}(\tilde{\gamma}_{i},\tilde{\gamma}_{i})=0$ since $i=j_c,$ $\tilde{\gamma}_{j_{c}}=\tilde{\gamma}_{i}$ and $\tilde{\beta}_{j_ci}=0.$ Therefore the summand, for $c \in U_i,$ is equal to  ${\rm Tr}_k(-\ell i_{m,r}(\varepsilon_{j_c,c}))={\rm Tr}_k(-\ell i_{m,r}(\varepsilon_{i,c})).$ Since $\varepsilon_{i,c}=0,$ for all but a finite number of $c \in U_i,$ we are done.

We  now show that the expression makes sense and is independent of all the choices. Note that there are many of them.

\begin{theorem}\label{mainthm}
For every $m<r<2m,$ the above formula (\ref{reg formula in general})  gives a well-defined regulator map 
$
\rho_{m,r}: \check{{\rm Z}} ^{3}(U_{\bigcdot}, \underline{\Gamma}''_{B}(C,\mathcal{P},3)) \to k,$  independent of all the choices.  This map vanishes on the coboundaries and hence induces the regulator map 
$$
\rho_{m,r}: {\rm H}^{3}_{B}(C,\mathbb{Q}(3))\to k
$$
of $\star$-weight $r.$ 
\end{theorem}

\begin{proof} We first show the independence of the definition from the various choices. For readability, we separate these into parts.

{\it Independence of the choice of $j_c$ and the liftings $\tilde{\beta}_{j_ci}$ and $\tilde{\gamma}_{j_c}.$}  Suppose that we choose a  different $j_c'$ with $c \in U_{j_c'};$ a different lifting $\tilde{\pazocal{A}}'_{c}$ of $\hat{\pazocal{O}}_{C,c},$ together with $\tilde{\mathfrak{c}}'$ as above; a $\tilde{\mathfrak{c}}'$-good lifting $\tilde{\gamma}'_{j_c'}$ of $\gamma_{j_c'}$ to $\tilde{\pazocal{A}}_{c}';$  
and a lifting $\tilde{\beta}'_{j'_ci}$ of 
$\beta_{j'_ci}$ to $\tilde{\pazocal{A}}_{\eta}.$
Since $\tilde{\pazocal{A}}_{c} \simeq k(c)[[\tilde{\mathfrak{s}}]]_{\infty},$ where $\tilde{\mathfrak{s}}$ is a choice of a  uniformizer associated to $\tilde{\mathfrak{c}}$ and similarly for $\tilde{\pazocal{A}}_{c}',$ we choose and fix a $k_{\infty}$-algebra isomorphism between $\tilde{\pazocal{A}}_{c}$ and $\tilde{\pazocal{A}}_{c}'$ which is identity modulo $(t^m)$ and which sends $\tilde{\mathfrak{s}}$ to $\tilde{\mathfrak{s}}'.$ This last condition is possible to impose since both $\tilde{\mathfrak{s}}$ and $\tilde{\mathfrak{s}}'$ lift $\mathfrak{s}$ by assumption. Below we  identify these two algebras using this isomorphism. 

 We need to compare the two expressions 
\begin{align}\label{reg indep loc ex 1}
   \ell _{m,r}(res_{\tilde{\mathfrak{c}}}\tilde{\gamma}_{j_c})-\ell i _{m,r} (\varepsilon_{j_c,c})+res_{c}\omega_{m,r}(\tilde{\gamma}_{i\eta}-\delta(\tilde{\beta}_{j_ci}),\tilde{\gamma}_{j_c}) 
\end{align}
and 
\begin{align}\label{reg indep loc ex 2}
 \ell _{m,r}(res_{\tilde{\mathfrak{c}}'}\tilde{\gamma}'_{j'_c})-\ell i _{m,r} (\varepsilon_{j_c',c})+res_{c}\omega_{m,r}(\tilde{\gamma}_{i\eta}-\delta(\tilde{\beta'}_{j'_ci}),\tilde{\gamma}'_{j'_c}).   
\end{align}

By linearity we have 
$$
res_{c}\omega_{m,r}(\tilde{\gamma}_{i\eta}-\delta(\tilde{\beta'}_{j'_ci}),\tilde{\gamma}'_{j'_c})-res_{c}\omega_{m,r}(\tilde{\gamma}_{i\eta}-\delta(\tilde{\beta}_{j_ci}),\tilde{\gamma}_{j_c}) =res_{c}\omega_{m,r}(\tilde{\gamma}_{j_c}-\tilde{\gamma}'_{j'_c},\delta(\tilde{\beta'}_{j'_ci}-\tilde{\beta}_{j_ci})). 
$$
Let $\tilde{\beta}_{j_c ' j_c}$ be a $\tilde{\mathfrak{c}}$-good lifting of $\beta_{j_c'j_c}$ to $\tilde{\pazocal{A}}_{c}.$  Since $\beta_{j_c'j_c}$ itself is $\mathfrak{c}$-good such a lifting exists.    We have the identity $\beta_{j_c'j_c}=\beta_{j_c'i}-\beta_{j_c i}$ on $U_{ij_c j_c '},$ which might not contain $c,$ but does of course contain the generic point $\eta.$ We deduce that 
$\tilde{\beta}_{j_c ' j_c, \eta}$ and $\tilde{\beta'}_{j'_ci,\eta}-\tilde{\beta}_{j_ci,\eta}$ have the same reduction $\beta_{j_c'j_c,\eta}.$ Now by Corollary \ref{corollary value of omega on a boundary},  we conclude that 
$$
res_{c}\omega_{m,r}(\delta(\tilde{\beta}_{j_c ' j_c, \eta}),\delta(\tilde{\beta'}_{j'_ci,\eta}-\tilde{\beta}_{j_ci,\eta}))=0.
$$
This implies that, using  transitivity and linearity, we have:  
$$
res_{c}\omega_{m,r}(\tilde{\gamma}_{j_c}-\tilde{\gamma}'_{j'_c},\delta(\tilde{\beta'}_{j'_ci}-\tilde{\beta}_{j_ci}))=res_{c}\omega_{m,r}(\tilde{\gamma}_{j_c}-\tilde{\gamma}'_{j'_c},\delta(\tilde{\beta}_{j_c ' j_c, \eta}))=res_{c}\omega_{m,r}(\tilde{\gamma}_{j_c}-\delta(\tilde{\beta}_{j_c ' j_c, \eta}),\tilde{\gamma}'_{j'_c}).
$$
In this expression,   $\tilde{\gamma}_{j_c}-\delta(\tilde{\beta}_{j_c ' j_c})$ is a $\tilde{\mathfrak{c}}$-good lifting to $\tilde{\pazocal{A}}_{c}$ and $\tilde{\gamma}'_{j'_c}$ is a $\tilde{\mathfrak{c}}'$-good lifting to $\tilde{\pazocal{A}}_{c}'.$   Then Proposition \ref{general residue prop} implies that 
\begin{align}\label{regulator indep loc}
res_{c}\omega_{m,r}(\tilde{\gamma}_{j_c}-\delta(\tilde{\beta}_{j_c ' j_c}),\tilde{\gamma}'_{j'_c})=\ell_{m,r}(res_{\tilde{\mathfrak{c}}}(\tilde{\gamma}_{j_c}-\delta(\tilde{\beta}_{j_c ' j_c})))-\ell _{m,r}(res_{\tilde{\mathfrak{c}}'}(\tilde{\gamma}'_{j'_c})).
\end{align}
On the other hand, 
$$
\ell _{m,r}(res_{\tilde{\mathfrak{c}}}(\delta(\tilde{\beta}_{j_c ' j_c})))=\ell _{m,r}(\delta(res_{\tilde{\mathfrak{c}}}(\tilde{\beta}_{j_c ' j_c}) ) )=\ell i_{m,r}(res_{c}(\beta_{j_c ' j_c})),
$$
by the definition of $\ell i_{m,r}.$ Since by assumption $res_c(\beta_{j_c ' j_c})=\varepsilon_{j_c,c}-\varepsilon_{j_c',c},$ we can rewrite the right hand side of (\ref{regulator indep loc}) as 
$$
\ell_{m,r}(res_{\tilde{\mathfrak{c}}}(\tilde{\gamma}_{j_c}))- \ell_{m,r}(res_{\tilde{\mathfrak{c}}'}(\tilde{\gamma}'_{j'_c}))-\ell i_{m,r}(\varepsilon_{j_c,c})
+\ell i_{m,r}(\varepsilon_{j_c',c}).$$ 
Combining all of the above, we see that the last expression is equal to the difference 
$$
res_{c}\omega_{m,r}(\tilde{\gamma}_{i\eta}-\delta(\tilde{\beta'}_{j'_ci}),\tilde{\gamma}'_{j'_c})-res_{c}\omega_{m,r}(\tilde{\gamma}_{i\eta}-\delta(\tilde{\beta}_{j_ci}),\tilde{\gamma}_{j_c}),
$$
which implies the equality of the two expressions (\ref{reg indep loc ex 1}) and (\ref{reg indep loc ex 2}) and thus proves the independence we were looking for. 

{\it Independence of the choice of $i$ and the liftings $\tilde{\gamma}_{i\eta}$ and $\tilde{\beta} _{j_c i}.$} 
Let us choose an $i',$ a lifting $\tilde{\pazocal{A}}'_{\eta}$ of $\pazocal{O}_{C,\eta}$  and liftings $\tilde{\gamma}'_{i'\eta}$ and $\tilde{\beta}' _{j_c i'}$ to $\tilde{\pazocal{A}}'_{\eta},$ for each $c \in |C|.$ 

We need to compare 

\begin{align}\label{reg global indep ex 1}
    \sum _{c \in |C|}{\rm Tr}_k\big(\ell _{m,r}(res_{\tilde{\mathfrak{c}}}\tilde{\gamma}_{j_c})-\ell i _{m,r} (\varepsilon_{j_c,c})+res_{c}\omega_{m,r}(\tilde{\gamma}_{i\eta}-\delta(\tilde{\beta}_{j_ci}),\tilde{\gamma}_{j_c})\big)\end{align}
and 
\begin{align}\label{reg global indep ex 2}
    \sum _{c \in |C|}{\rm Tr}_k\big(\ell _{m,r}(res_{\tilde{\mathfrak{c}}}\tilde{\gamma}_{j_c})-\ell i _{m,r} (\varepsilon_{j_c,c})+res_{c}\omega_{m,r}(\tilde{\gamma}'_{i'\eta}-\delta(\tilde{\beta}' _{j_c i'}),\tilde{\gamma}_{j_c})\big).
    \end{align}

 The difference between the above expressions is  
 $$
  \sum _{c \in |C|}{\rm Tr}_kres_{c}\omega_{m,r}(\tilde{\gamma}'_{i'\eta}-\delta(\tilde{\beta}' _{j_c i'}),\tilde{\gamma}_{i\eta}-\delta(\tilde{\beta}_{j_ci})).
 $$
 Choosing an isomorphism $\tilde{\pazocal{A}}_{\eta}\simeq \tilde{\pazocal{A}}'_{\eta}$ of $k_{\infty}$-algebras which lifts the given one modulo $(t^m),$ we identify $\tilde{\pazocal{A}}_{\eta}$ and  $\tilde{\pazocal{A}}'_{\eta}.$ The above sum can then be rewritten as: 
 $$
 \sum _{c \in |C|}{\rm Tr}_kres_{c}\omega_{m,r}(\tilde{\gamma}'_{i'\eta}-\tilde{\gamma}_{i\eta},\delta(\tilde{\beta}' _{j_c i'}-\tilde{\beta}_{j_ci})).
 $$
As in the above argument since $\tilde{\beta}_{ii'}$ has the same reduction modulo $(t^m)$ as $\tilde{\beta}' _{j_c i'}-\tilde{\beta}_{j_ci}$ for any $j_c,$ we have $res_{c}\omega_{m,r}(\delta(\tilde{\beta}' _{j_c i'}-\tilde{\beta}_{j_ci}),\delta(\tilde{\beta}_{ii'}))=0$ by Corollary \ref{corollary value of omega on a boundary}. So  we can rewrite the above sum as: 
$$
\sum _{c \in |C|}{\rm Tr}_kres_{c}\omega_{m,r}(\tilde{\gamma}'_{i'\eta}-\tilde{\gamma}_{i\eta},\delta(\tilde{\beta}_{ii'})).
$$
 Choosing a  splitting of $\tilde{\pazocal{A}}_{\eta},$ we identify this algebra with $(\underline{\tilde{\pazocal{A}}}_{\eta})_{\infty}=(\pazocal{O}_{\underline{C},\eta})_{\infty}.$ Using this identification,  the last expression is the sum of residues of the  meromorphic 1-form $\Omega_{m,r}(\tilde{\gamma}'_{i'\eta}-\tilde{\gamma}_{i\eta}-\delta(\tilde{\beta}_{ii'}))$  on $\underline{C}$ and therefore is equal to 0.

 {\it Vanishing on coboundaries.} Suppose that we start with sections 
 $$
 \alpha_{i} \in (B_{2}(\pazocal{O}_C, \underline{\mathcal{P}}) \otimes (\pazocal{O}_{C},\underline{\mathcal{P}}) ^{\times})(U_{i}),
 $$ 
 for all $i \in I.$ Then we need to show that the value of the regulator on the data $$
 (\{\gamma_i\}_{i\in I}, \{ \varepsilon_{i,c}| i\in I, \, c\in U_i\}, \{ \beta_{ij}\}_{i, j \in I})
 $$   
 is 0. Here $\gamma_i:=\delta(\alpha_i),$ $\varepsilon_{i,c}:=res_{\mathfrak{c}}(\alpha_i)$ and $\beta_{ij}:=\alpha_j|_{U_{ij}}-\alpha_i|_{U_{ij}}.$
 
We fix an  $i \in I$ and $j_{c} \in I,$ with $c \in U_{j_c},$ for every $c \in |C|;$ and local and generic liftings $\tilde{\pazocal{A}}_{c},$ and $\tilde{\pazocal{A}}_{\eta}$ of the curve, as above,  together with liftings $\tilde{\mathfrak{c}}$ of $\mathfrak{c}$ to $\tilde{\pazocal{A}}_{c}.$ We need to choose liftings of the data in order to compute the value of the regulator on the above element. 
 
We  choose a lifting $\tilde{\alpha} _{i\eta}$  of $\alpha_{i\eta}$ to $\tilde{\pazocal{A}}_{\eta}$ and let $\tilde{\gamma}_{i\eta}:=\delta(\tilde{\alpha}_{i\eta}).$ For each $c \in |C|,$ we choose a $\tilde{\mathfrak{c}}$-good lifting $\tilde{\alpha}_{j_c}$ of $\alpha_{j_c}$ and let $\tilde{\gamma}_{j_c}:=\delta(\tilde{\alpha}_{j_c}).$ Finally, we choose an arbitrary lifting $\tilde{\alpha}_{j_c\eta}$ of $\alpha_{j_c\eta}$ to $\tilde{\pazocal{A}}_{\eta},$ for every $c \in |C|,$ and let $\tilde{\beta}_{j_ci,\eta}:=\tilde{\alpha}_{i\eta}-\tilde{\alpha}_{j_c\eta}.$
Then the value of the regulator (\ref{reg formula in general}) is the sum of traces of the terms: 
\begin{align*}
    &\ell _{m,r}(res_{\tilde{\mathfrak{c}}}\tilde{\gamma}_{j_c})-\ell i _{m,r} (\varepsilon_{j_c,c})+res_{c}\omega_{m,r}(\tilde{\gamma}_{i\eta}-\delta(\tilde{\beta}_{j_ci,\eta}),\tilde{\gamma}_{j_c})\\
    =&\ell _{m,r}(res_{\tilde{\mathfrak{c}}}\delta(\tilde{\alpha}_{j_c}))-\ell i _{m,r} (\varepsilon_{j_c,c})+res_{c}\omega_{m,r}(\delta(\tilde{\alpha}_{j_c\eta}),\delta(\tilde{\alpha}_{j_c}))\\
    =&\ell _{m,r}(res_{\tilde{\mathfrak{c}}}\delta(\tilde{\alpha}_{j_c}))-\ell i _{m,r} (\varepsilon_{j_c,c})
\end{align*}
by Corollary \ref{corollary value of omega on a boundary}. Since $res_{\tilde{\mathfrak{c}}}\delta(\tilde{\alpha}_{j_c})=\delta(res_{\tilde{\mathfrak{c}}}\tilde{\alpha}_{j_c}),$ we have $\ell _{m,r}(res_{\tilde{\mathfrak{c}}}\delta(\tilde{\alpha}_{j_c}))=\ell _{m,r}(\delta(res_{\tilde{\mathfrak{c}}}\tilde{\alpha}_{j_c})).$ By the definition of $\ell i_{m,r},$ we have $\ell _{m,r}(\delta(res_{\tilde{\mathfrak{c}}}\tilde{\alpha}_{j_c}))=\ell i_{m,r}(res_{\mathfrak{c}}\alpha_{j_c})=\ell i_{m,r}(\varepsilon_{j_c,c}).$ This implies that all the summands in the formula for the regulator (\ref{reg formula in general}) are 0 finishing the proof of the theorem. 
 \end{proof}

    \subsection{Infinitesimal Chow Dilogarithm} Specializing  the above construction to global sections of $\Lambda ^3 (\pazocal{O}_{C},\underline{\mathcal{P}})^{\times}$ gives us the generalization of the infinitesimal Chow dilogarithm in \cite{unv3} to higher moduli.  

Suppose that we start with $\gamma \in \Gamma(\underline{C},\Lambda ^3 (\pazocal{O}_C,\underline{\mathcal{P}})^{\times}).$ Specializing the construction in the previous section,  we have  $\rho_{m,r}(\gamma) \in k,$ which can be computed as follows.  

Choose a  lifting  $ \tilde{\pazocal{A}}_{\eta}/k_{\infty}$ of  $\pazocal{O}_{C,\eta}$  and local  liftings $\tilde{\pazocal{A}}_{c}$  of $\hat{\pazocal{O}}_{C,c},$ for every $c \in |C|,$ together with liftings $\tilde{\mathfrak{c}}$ of $\mathfrak{c}.$ Choose an arbitrary lifting $\tilde{\gamma}_{\eta}$ of $\gamma_{\eta}$ to $\tilde{\pazocal{A}}_{\eta}$ and $\tilde{\mathfrak{c}}$-good liftings $\tilde{\gamma}_{c}$ of the germ of $\gamma$ at $c$ to $\tilde{\pazocal{A}}_c,$ for every $c \in |C|.$

By the definition in the previous section, we have 
\begin{align}\label{inf chow dilog}
   \rho_{m,r}(\gamma):= \sum_{c\in |C|}{\rm Tr}_k(l_{m,r}(res_{\tilde{\mathfrak{c}}}(\tilde{\gamma}_{c}) )+ res_{c} \omega_{m,r}(\tilde{\gamma}_{\eta},\tilde{\gamma}_{c})),
\end{align}
for every $m<r<2m.$ 

\begin{corollary}\label{cor inf dilog}
The definition in (\ref{inf chow dilog}) of the infinitesimal Chow dilogarithm of modulus $m$ and $\star$-weight $r$ gives a map 
$$
\rho_{m,r}: \Lambda ^3 (\pazocal{O}_{C},\underline{\mathcal{P}})^{\times} \to k,
$$
independent of all the choices and generalizing the construction in \cite{unv3} for $m=2$ and $r=3.$ 
\end{corollary}

Exactly as in \cite[Theorem 3.4.4]{unv3}, this infinitesimal Chow dilogarithm can be used to give a proof of an infinitesimal version of Goncharov's strong reciprocity conjecture for the curve $C/k_m.$

\subsection{ Invariants of cycles on $k_m$}

As before this construction gives us an invariant of cycles. For a cycle of modulus $m$ we expect the combination of $\rho_{m,r}$ for all $m<r<2m$ to be a complete set of invariants for the rational equivalence class of a cycle. 
For the appropriate, yet to be defined, Chow group ${\rm CH}^{2}(k_{m},3),$ we expect that our regulators $\rho_{m,r}$ to give complete  invariants for the infinitesimal part ${\rm CH}^{2}(k_{m},3)^{\circ}.$ This section also generalizes Park's construction of regulators \cite{park}, where the case of $r=m+1$ is dealt with.  Since this section is more or less a generalization of \cite[\textsection 4]{unv3}, we do not go into the details and explain certain constructions in a slighty alternate way.

First, let us recall  the  definition  of cubical higher Chow groups over a smooth $k$-scheme $X/k$ \cite{bloch}.  Let $\square_k:= \mathbb{P}^{1} _{k} \setminus \{ 1\}$ and $\square ^n _{k}$ the $n$-fold product of $\square_k$ with itself over $k, $ with the coordinate functions $y_1, \cdots, y_n.$ For  a smooth $k$-scheme $X,$ we let   $\square^n _{X} :=X \times_k \square_k ^n.$  A codimension 1 face of $\square^n _{X}$ is a divisor $F_{i} ^a$  of the form $y_{i}=a,$ for $1\leq i \leq n,$ and $a \in \{0,\infty \}.$ A face of $\square^n _{X}$ is either the whole scheme $\square^n _{X}$ or an arbitrary intersection of codimension 1 faces. Let $\underline{z}^q (X, n)$ be the free abelian group on the set of codimension $q,$ integral, closed subschemes $Z \subseteq  \square^n _{X}$ which are {\it admissable}, i.e.  which intersect each face properly on $\square^n _{X}.$ For each codimension one face $F_{i} ^a,$ and  irreducible $Z \in \underline{z} ^q (X, n)$, we let $\partial_i ^{a} (Z)$  be the cycle associated to the scheme $Z \cap F_{i} ^{a}.$ We let $\partial:= \sum_{i=1} ^n (-1)^n (\partial_i ^{\infty} - \partial_i ^0)$ on $\underline{z}^q (X, n),$ which gives a complex $(\underline{z}^q (X, \cdot),\partial).$ Dividing this complex by the subcomplex of degenerate cycles, we obtain Bloch's higher Chow group complex whose  homology ${\rm CH}^q (X, n):= {\rm H}_n (z^q (X, \cdot))$ is the higher Chow group of $X$.

In order to work with a candidate for Chow groups of cycles on $k_m$, we need to work with cycles over $k_{\infty}$ which have a certain finite reduction property. Let $\overline{\square}_{k}:= \mathbb{P}^{1} _{k},$  $\overline{\square}_{k} ^{n},$ the $n$-fold product of  $\overline{\square}_{k} $ with itself over $k,$ and  $\overline{\square}_{k_{\infty}} ^{n} :=\overline{\square}_{k} ^{n} \times _k k_{\infty}.$ We define a subcomplex $\underline{z}^q _{f} (k_{\infty}, \cdot) \subseteq \underline{z}^q (k_{\infty}, \cdot)$, as  the subgroup generated by integral, closed subschemes $Z \subseteq \square_{k_{\infty}} ^n$ which are admissible in the above sense and have {\it finite reduction}, i.e. $\overline{Z}$  intersects each $s\times \overline{F}$ properly on $\overline{\square}_{k_\infty} ^n,$  for  every face $F$ of $\square^n_{k_{\infty}}.$ Here $s$ denotes the closed point of the spectrum of $k_{\infty}$ and for a subscheme $Y\subseteq \square_{k_{\infty}} ^n,$ $\overline{Y}$ denotes its closure in $\overline{\square}_{k_\infty} ^n.$ Modding out by degenerate cycles, we  have a complex $z^q_{f} (k_\infty, \cdot).$

Fix $2\leq m<r<2m.$ Let $\eta$ denote the generic point of the spectrum of $k_{\infty}.$ An irreducible cycle $p$ in $ \underline{z}_{f} ^2 (k_{\infty},2)$ is  given by a closed point $p_{\eta} $  of  $\square ^{2} _{\eta}$ whose closure $\overline{p}$  in $\overline{\square} ^{2} _{k_{\infty}}$ does not meet $(\{ 0,\infty\} \times \overline{\square}_{k_{\infty}}) \cup ( \overline{\square}_{k_{\infty}} \times \{0, \infty \}).$ Let $\tilde{p}$ denote the normalisation of $\overline{p} $ and $T$ denote the underlying set of  the closed fiber $\tilde{p}\times _{k_{\infty}}s$ of $\tilde{p}.$ For every $s' \in T,$ and $1\leq i,$  define $\ell_{\tilde{p},s',i}:\hat{\pazocal{O}}_{\tilde{p},s'} ^{\times } \to k(s')$ by the formula: 
$$
\ell_{\tilde{p},s',i}(y):=\frac{1}{i}res_{\tilde{p},s'} \frac{1}{t^i} d\log (y).$$ 

 Let
\begin{eqnarray}\label{defnl} 
\;\; \; l_{m,r}(p):=\sum _{s' \in T}{\rm Tr}_{k} \sum_{1\leq i \leq r-m}i \cdot (\ell_{\tilde{p},s', r-i}\wedge \ell _{\tilde{p}, s',i} )(y_1\wedge y_2).
\end{eqnarray}
Note the similarity with Definition \ref{defnlmr}.

\begin{definition}
We   define the regulator 
$
\rho _{m,r}: \underline{z}_{f} ^2 (k_{\infty},3) \to k
$ 
 as the composition $l_{m,r}  \circ \partial .$
\end{definition}
Exactly as in \cite{unv3}, one proves that the regulator above vanishes on boundaries and products, is alternating and has the same value on cycles which are congruent modulo $(t^m).$ We state only this last property, which is the most important one, in detail.  

 Suppose that $Z_i$ for $i=1,2$ are two irreducible cycles in $\underline{z}^{2} _{f} (k_{\infty},3).$ We say that $Z_{1}$ and $Z_{2}$ are equivalent modulo $t^m$ if the following condition $(M_{m})$ holds:
 
 (i) $\overline{Z}_{i}/k_{\infty}$ are smooth with $(\overline{Z}_i)_{s} \cup (\cup_{j,a} |\partial _j ^{a} Z_i|) $  a strict normal crossings divisor on $\overline{Z}_i.$
  
 and 
 
 (ii) $\overline{Z}_{1}|_{t^m}=Z_{2} |_{t^m}.$ 
 
 Then we have:

\begin{theorem}\label{modulus theorem}
If $Z_{i} \in \underline{z} _{f} ^{2}(k_{\infty},3),$ for $i=1,2,$  satisfy the condition $(M_{m}),$ for some $m\geq2,$ then they have the same infinitesimal regulator value: 
$$\rho_{m,r} (Z_{1})=\rho_{m,r}(Z_{2}),$$
for every $m<r<2m.$ 
\end{theorem}

\begin{proof} 
The proof is exactly as in \cite{unv3} and is based on Corollary \ref{cor inf dilog}.
 \end{proof}
As we remarked above, we expect the invariants $\rho_{m,r}$ for $m<r<2m$ to give a full set of invariants in the infinitesimal part of a  yet to be defined Chow group ${\rm CH}^2(k_{m},3).$

\end{document}